\documentclass[11pt,a4paper]{article}
\usepackage{modnatbib}
\usepackage{amssymb,amsmath,amsthm,epsfig,verbatim,pstricks,url}
\usepackage{tikz,pgfplots,tikz-3dplot}
\usepackage{setspace}
\usepackage{ifpdf}  
\newtheorem{thm}{\sc Theorem.}[section]
\newtheorem{lem}[thm]{\sc Lemma.}
\newtheorem{rem}[thm]{\sc Remark.}

\newenvironment{AMS}%
{{\upshape\bfseries AMS subject classifications. }\ignorespaces}{}
\newenvironment{keywords}{{\upshape\bfseries Key words. }\ignorespaces}{}

\newcommand{\bRplus}{{\mathbb R}_{>0}}
\newcommand{\bRgeq}{{\mathbb R}_{\geq 0}}

\newcommand{\RZ}{{\mathbb R} \slash {\mathbb Z}}
\newcommand{\bR}{{\mathbb R}}
\newcommand{\bN}{{\mathbb N}}

\newcommand{\ratio}{{\mathfrak r}}
\newcommand{\dH}[1]{\;{\rm d}{\mathcal{H}}^{#1}} 
\newcommand{\drho}{\;{\rm d}\rho}
\newcommand{\dt}{\;{\rm d}t}
\newcommand{\Vh}{\underline{V}^h}

\newcommand{\Vhpartialzero}{\underline{V}^h_{\partial_0}}

\newcommand{\Vpartialzero}{\underline{V}_{\partial_0}}

\newcommand{\id}{{\rm id}}

\newcommand{\dd}[1]{\frac{\rm d}{{\rm d}#1}}
\newcommand{\ddt}{\dd{t}}

\newcommand{\x}{\vec x}
\newcommand{\ek}{e}
\newcommand{\ttau}{\Delta t}

\def\epsilon{\varepsilon}

\begin{document}
\title{
A finite element error analysis for \\ axisymmetric mean curvature flow%
\thanks{John passed away on 30 June 2019, when this manuscript was nearly 
completed. We dedicate this article to his memory.}
}
\author{John W. Barrett\footnotemark[2] \and 
        Klaus Deckelnick\footnotemark[3]\ \and 
        Robert N\"urnberg\footnotemark[2]}

\renewcommand{\thefootnote}{\fnsymbol{footnote}}
\footnotetext[2]{Department of Mathematics, 
Imperial College London, London, SW7 2AZ, UK}
\footnotetext[3]{Institut f\"ur Analysis und Numerik,
Otto-von-Guericke-Universit\"at Magdeburg, 
39106 Magdeburg, Germany}

\date{}

\maketitle

\begin{abstract}
We consider the numerical approximation of axisymmetric mean curvature flow
with the help of linear finite elements. In the case of a closed 
genus-1 surface, we derive optimal error bounds with respect to the $L^2$-- 
and $H^1$--norms for a fully discrete approximation. 
We perform convergence experiments to
confirm the theoretical results, and also present numerical simulations for
some genus-0 and genus-1 surfaces, including for the Angenent torus.
\end{abstract} 


\renewcommand{\thefootnote}{\arabic{footnote}}

\setcounter{equation}{0}
\section{Introduction} \label{sec:intro}

Mean curvature flow is one of the simplest prototypes for a geometric evolution
equation, and it has been studied extensively in differential geometry and
numerical analysis over the last few decades. 
For a family of closed hypersurfaces
$(\mathcal{S}(t))_{t \geq 0} \subset \bR^3$ 
this geometric evolution law is given by
\begin{equation} \label{eq:mcfS}
\mathcal{V}_{\mathcal{S}} = k_m\qquad\text{ on } \mathcal{S}(t)\,,
\end{equation}
where $\mathcal{V}_{\mathcal{S}}$ denotes the normal velocity of 
$\mathcal{S}(t)$ in the direction of the normal $\vec\nu_{\mathcal{S}(t)}$,
and $k_m$ is the mean curvature of $\mathcal{S}(t)$, i.e.\ the sum of its
principal curvatures. 
For an introduction to mean curvature flow and important results we refer to 
\cite{Mantegazza11}.

The approximation of mean curvature flow for two-dimensional surfaces with the
help of parametric finite elements goes back to the seminal work by
\cite{Dziuk91}, and other methods have since been proposed, see e.g.\
\cite{gflows3d,ElliottF17}. The first convergence proof for a parametric method has only very recently been obtained in \cite{KovacsLL19} for
a system that employs the position, the normal vector and the mean curvature as variables. Optimal $H^1$--error bounds are proven for a spatial
discretization by surface finite elements of order $k \geq 2$ and backward difference formulae for time integration. For surfaces that can be
written as a graph, error bounds have previously been shown in 
\cite{DeckelnickD95b,DeckelnickD00}. We refer to the review
articles \cite{DeckelnickDE05,bgnreview} 
for further information on the numerical approximation
of mean curvature flow and related geometric evolution equations, where the 
former article also surveys level set and phase field methods
not discussed in this introduction.

In this paper we consider mean curvature flow in an axisymmetric setting. We set
\[
I = \RZ\,, \text{ with } \partial I = \emptyset\,,\quad \text{or}\quad
I = (0,1)\,, \text{ with } \partial I = \{0,1\}\,,
\]
and let $\vec x(t) : \overline I \to \bRgeq\times \bR$ 
parameterize $\Gamma(t)$, which is the generating curve of a
surface $\mathcal{S}(t)$ 
that is axisymmetric with respect to the $x_2$--axis, see
Figure~\ref{fig:sketch}. 
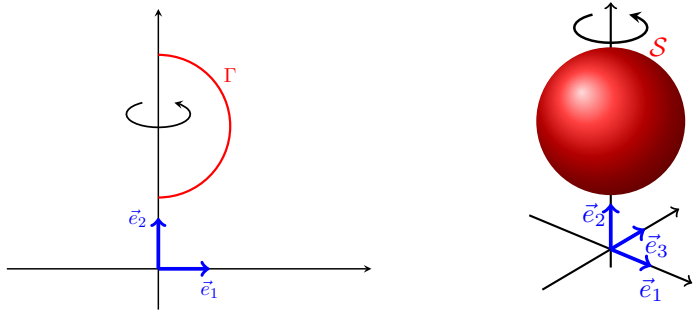
\begin{figure}
\center
\newcommand{\AxisRotator}[1][rotate=0]{%
    \tikz [x=0.25cm,y=0.60cm,line width=.2ex,-stealth,#1] \draw (0,0) arc (-150:150:1 and 1);%
}
\newcommand{\AxisRotatorSmall}[1][rotate=0]{%
    \tikz [x=0.2cm,y=0.48cm,line width=.2ex,-stealth,#1] \draw (0,0) arc (-150:150:1 and 1);%
}
\begin{tikzpicture}[every plot/.append style={very thick}, scale = 0.7]
\begin{axis}[axis equal,axis line style=thick,axis lines=center, xtick style ={draw=none}, 
ytick style ={draw=none}, xticklabels = {}, 
yticklabels = {}, 
xmin=-0.2, xmax = 0.8, ymin = -0.4, ymax = 2.55]
after end axis/.code={  
   \node at (axis cs:0.0,1.5) {\AxisRotator[rotate=-90]};
   \draw[blue,->,line width=2pt] (axis cs:0,0) -- (axis cs:0.5,0);
   \draw[blue,->,line width=2pt] (axis cs:0,0) -- (axis cs:0,0.5);
   \node[blue] at (axis cs:0.5,-0.2){$\vec\ek_1$};
   \node[blue] at (axis cs:-0.2,0.5){$\vec\ek_2$};
   \draw[red,very thick] (axis cs: 0,0.7) arc[radius = 70, start angle= -90, end angle= 90];
   \node[red] at (axis cs:0.7,1.9){$\Gamma$};
}
\end{axis}
\end{tikzpicture} \qquad \qquad
\tdplotsetmaincoords{120}{50}
\begin{tikzpicture}[scale=1.4, tdplot_main_coords,axis/.style={->},thick]
\draw[axis] (-1, 0, 0) -- (1, 0, 0);
\draw[axis] (0, -1, 0) -- (0, 1, 0);
\draw[axis] (0, 0, -0.2) -- (0, 0, 2.7);
\draw[blue,->,line width=1.4pt] (0,0,0) -- (0,0.5,0) node [below] 
{\small{$\vec\ek_1$}};
\draw[blue,->,line width=1.4pt] (0,0,0) -- (0,0.0,0.5);
\draw[blue,->,line width=1.4pt] (0,0,0) -- (0.5,0.0,0);
\node[blue] at (0.2,0.4,0.1){\small{$\vec\ek_3$}};
\node[blue] at (0,-0.2,0.3){\small{$\vec\ek_2$}};
\node[red] at (0.7,0,1.9){\small{$\mathcal{S}$}};
\node at (0.0,0.0,2.4) {\AxisRotatorSmall[rotate=-90]};

\tdplottransformmainscreen{0}{0}{1.4}
\shade[tdplot_screen_coords, ball color = red] (\tdplotresx,\tdplotresy) circle (0.7);
\end{tikzpicture}
\caption{Sketch of $\Gamma$ and $\mathcal{S}$, as well as 
the unit vectors $\vec\ek_1$, $\vec\ek_2$ and $\vec\ek_3$.}
\label{fig:sketch}
\end{figure}%
Here we allow $\Gamma(t)$ to be either a closed curve, parameterized over
$\RZ$, which corresponds to $\mathcal{S}(t)$ being a genus-1 surface
without boundary.
Or $\Gamma(t)$ may be an open curve, parameterized over $[0,1]$,
which corresponds to $\mathcal{S}(t)$ being a genus-0 surface, e.g.\ a sphere.
Then the two endpoints of $\Gamma(t)$ are attached to the $x_2$--axis, 
on which they can freely move up and down. 
In particular, we always assume that, for all $t \in [0,T]$,
\begin{subequations} 
\begin{align} 
\vec x(\rho,t)\cdot\vec\ek_1 & > 0 \quad 
\forall\ \rho \in \overline I\setminus \partial I\,,\label{eq:xpos} \\
\vec x(\rho,t)\cdot\vec\ek_1 &= 0 \quad 
\forall\ \rho \in \partial I\,.\label{eq:axibc} 
\end{align}
\end{subequations}
Let us denote by $\vec\tau$ and $\vec\nu$ a unit tangent and a unit normal to $\Gamma(t)$, respectively.
It was shown in \cite{aximcf} that mean curvature flow, \eqref{eq:mcfS}, 
for an axisymmetric surface can be formulated as 
\begin{subequations} 
\begin{alignat}{2} 
\vec x_t\cdot\vec\nu & = \varkappa - 
\frac{\vec\nu\cdot\vec\ek_1}{\vec x\cdot\vec\ek_1} \quad && \text{in } I
\times (0,T]\,,\label{eq:xt} \\
\vec x_\rho \cdot\vec\ek_2 & = 0 \quad && \text{on } \partial I \times (0,T]\,,
\label{eq:bc}
\end{alignat}
\end{subequations}
where $\varkappa = \vec\varkappa\cdot\vec\nu$ and 
$\vec\varkappa = \frac1{|\vec x_\rho|} 
\left( \frac{\vec x_\rho}{|\vec x_\rho|} \right)_\rho$ is the curvature vector.
In particular, we observe that a solution of the system  
\begin{equation} \label{eq:Dstrongnew}
\vec x_t 
- \frac1{|\vec x_\rho|} \left( \frac{\vec x_\rho}{|\vec x_\rho|} \right)_\rho 
+ \frac{\vec\nu\cdot\vec\ek_1}{\vec x\cdot\vec\ek_1}\,\vec\nu = \vec 0
\end{equation}
will satisfy (\ref{eq:xt}). Note that the leading part in the above problem
is the same as in the curve shortening flow
\begin{equation} \label{eq:csf}
\vec x_t 
- \frac1{|\vec x_\rho|} \left( \frac{\vec x_\rho}{|\vec x_\rho|} \right)_\rho 
= \vec 0\,.
\end{equation}
Hence it is natural to use numerical methods designed for (\ref{eq:csf}), 
in order to approximate axisymmetric mean curvature flow. 
Optimal error bounds for a semi--discrete scheme approximating (\ref{eq:csf}) 
have been obtained in \cite{Dziuk94}, and various schemes based on 
(\ref{eq:Dstrongnew}) are considered in \cite{aximcf}, although no error 
analysis is given. A difficulty of the approaches using
(\ref{eq:csf}) and (\ref{eq:Dstrongnew}) lies in the fact that, in view of 
\[
\frac1{|\vec x_\rho|} \left( \frac{\vec x_\rho}{|\vec x_\rho|} \right)_\rho 
= \frac1{|\vec x_\rho|^2} \left[ \vec x_{\rho\rho} 
- (\vec x_{\rho\rho} \cdot\vec \tau)\,\vec \tau \right] ,
\]
the resulting systems are degenerate in the tangential direction. 
One way to address this problem is DeTurck's trick, 
which essentially consists in removing this degeneracy by introducing an 
additional tangential motion via a suitable reparameterization. 
In the case of the curve shortening flow, it is natural to consider the 
system 
\begin{equation}  \label{eq:csfdeTurck}
\vec x_t - \frac{\vec x_{\rho \rho}}{|\vec x_\rho|^2} = \vec 0\,,
\end{equation}
for which a semi--discretization by linear finite elements was analyzed in
\cite{DeckelnickD95}. A whole family of schemes based on DeTurck's trick 
were introduced in \cite{ElliottF17}, both for curves and surfaces. 
It turns out that for curves the analysis of \cite{DeckelnickD95} 
can be generalized to these methods. \\
The application of DeTurck's trick to our problem amounts to replacing 
(\ref{eq:Dstrongnew}) by the system
\begin{equation} \label{eq:newsystem} 
\vec x_t - \frac{\vec x_{\rho \rho}}{|\vec x_\rho|^2} +
\frac{\vec\nu\cdot\vec\ek_1}{\vec x\cdot\vec\ek_1}\,\vec\nu = \vec 0\,.
\end{equation}
If we multiply this equation by 
$\vec x \cdot \vec\ek_1\,| \vec x_\rho |^2$ and note that $\vec\ek_1 = (\vec\nu\cdot\vec\ek_1)\,\vec\nu +
|\vec x_\rho|^{-2}\,(\vec x_\rho\cdot\vec\ek_1)\,\vec x_\rho$, we are led to the following system of PDEs
\begin{equation} \label{eq:DDstrong3}
\vec x\cdot\vec\ek_1\,|\vec x_\rho|^2\,\vec x_t - 
((\vec x\cdot\vec\ek_1)\,\vec x_{\rho})_\rho +
|\vec x_\rho|^2\,\vec\ek_1 = \vec 0\,,
\end{equation}
on which we base the numerical scheme to be analyzed in this paper. 
In Section~\ref{sec:weak} we derive a weak formulation of 
\eqref{eq:DDstrong3} together with a natural fully discrete approximation
using linear finite elements in space and a backward Euler scheme in time. 
The method is semi--implicit and hence requires the solution of a linear 
problem at each time step. In Section~\ref{sec:error} we prove the main 
result of our paper, which are optimal error bounds both in $H^1$ and in $L^2$
in the case $I = \RZ$, i.e.\ for genus-1 surfaces. Let us remark that on using 
our techniques, it is straightforward to
also obtain optimal $L^2$--error bounds for a 
fully discrete approximation of the curve shortening flow, something that
to the best of our knowledge has not yet appeared in the literature.
In Section~\ref{sec:alt} we briefly discuss an alternative formulation of
axisymmetric mean curvature flow and the associated discretization.
This alternative formulation has the advantage of allowing for an unconditional
stability estimate for the fully discrete approximation. However, in
contrast to the scheme derived in Section~\ref{sec:weak}, it appears that the
formulation is only well-posed in the case $I = \RZ$.
Finally,
in Section~\ref{sec:nr} we perform some numerical experiments to investigate
the robustness and the accuracy of the introduced scheme,
and to study some phenomena of interest in differential geometry, 
for example the Angenent torus.

We end this section with a few comments about notation. 
The $L^2$--inner product on $I$ is denoted by $(\cdot,\cdot)$.
We adopt the standard notation for Sobolev spaces, denoting the norm of
$W^{\ell,p}(I)$ ($\ell \in \bN_0$, $p \in [1, \infty]$)
by $\|\cdot \|_{\ell,p}$ and the 
semi-norm by $|\cdot |_{\ell,p}$. For
$p=2$, $W^{\ell,2}(I)$ will be denoted by
$H^{\ell}(I)$ with the associated norm and semi-norm written as,
respectively, $\|\cdot\|_\ell$ and $|\cdot|_\ell$.
The above are naturally extended to vector functions, and we will write 
$[W^{\ell,p}(I)]^2$ for a vector function with two components.
In addition, we adopt the standard notation $W^{\ell,p}(a,b;B)$
($\ell \in \bN$, $p \in [1, \infty]$, $(a,b)$ an interval in $\bR$, 
$B$ a Banach space) for time dependent spaces
with norm $\|\cdot\|_{W^{\ell,p}(a,b;B)}$.
Once again, we write $H^{\ell}(a,b;B)$ if $p=2$. 
Furthermore, $C$ denotes a generic positive constant independent of 
the mesh parameter $h$ and the time step size $\ttau$, see below.
For later use we recall the well--known Sobolev embedding
\begin{equation} \label{eq:sobolev}
|f|_{0,\infty} \leq C\,\|f\|_1 \qquad \forall\ f \in H^1(I)\,.
\end{equation}

\setcounter{equation}{0}
\section{Weak formulation and finite element discretization} \label{sec:weak}

Let 
$\Vpartialzero = \{ \vec\eta \in [H^1(I)]^2 : \vec\eta\cdot\vec\ek_1 = 0
\text{ on } \partial I\}$.
A weak formulation for (\ref{eq:DDstrong3}) is given as
follows.

\noindent
$(\mathcal P)$
Let $\vec x(0) \in \Vpartialzero$. For $t \in (0,T]$
find $\vec x(t) \in \Vpartialzero$ such that
\begin{equation} \label{eq:DD}
 \left( (\vec x \cdot\vec\ek_1)\,\vec x_t,\vec\eta\,|\vec x_\rho|^2\right)
+ \left( (\vec x\cdot\vec\ek_1) \,\vec x_\rho,\vec\eta_\rho \right)
+ \left( \vec\eta\cdot\vec\ek_1, |\vec x_\rho|^2 \right)
 = 0
\qquad \forall\ \vec\eta \in \Vpartialzero\,.
\end{equation}
It can be shown that the weak formulation (\ref{eq:DD}), despite the 
degenerate weight in the second term, weakly enforces the boundary
condition (\ref{eq:bc}), see \citet[Appendix~A]{aximcf} for the necessary 
techniques.

In order to define our finite element approximation of $(\mathcal P)$, 
let $[0,1]=\bigcup_{j=1}^J I_j$, $J\geq3$, be a
decomposition of $[0,1]$ into intervals given by the nodes $q_j$,
$I_j=[q_{j-1},q_j]$. 
For simplicity 
we assume that the subintervals form an equipartitioning of $[0,1]$,
i.e.\ that 
\begin{equation*} 
q_j = j\,h\,,\quad \mbox{with}\quad h = \frac 1J\,,\qquad j=0,\ldots, J\,.
\end{equation*}
Clearly, if $I=\RZ$ we identify $0=q_0 = q_J=1$.

We define the finite element spaces
$V^h = \{\chi \in C(\overline I) : \chi\!\mid_{I_j} 
\mbox{ is affine},\ j=1,\ldots, J\}$,
$\Vh = [V^h]^2$ and $\Vhpartialzero = \Vh \cap \Vpartialzero$.
Let $\{\chi_j\}_{j=j_0}^J$ denote the standard basis of $V^h$,
where $j_0 = 0$ if $I = (0,1)$ and $j_0 = 1$ if $I=\RZ$.
For later use, we let $\pi^h:C(\overline I)\to V^h$ 
be the standard interpolation operator at the nodes $\{q_j\}_{j=0}^J$.
In addition, we introduce $Z^h= \{ \chi \in L^\infty(I): \chi\!\mid_{I_j} 
\mbox{ is constant},\ j=1,\ldots, J\}$ and define 
$P^h:L^1(I) \to Z^h$ by
\begin{equation} \label{eq:defph}
(P^h f)_{| I_j} = \frac1{h}\, \int_{I_j} f \drho\,, \quad j=1,\ldots,J.
\end{equation}
It is well--known that for 
$k \in \{ 0,1 \}, \ell \in \{ 1,2 \}, p \in [2,\infty]$ 
it holds that
\begin{subequations}
\begin{alignat}{2}
h^{\frac 1p - \frac 1r}\,| \eta |_{0,r} 
+ h\,| \eta |_{1,p} & \leq C\,| \eta |_{0,p} 
\qquad && \forall\ \eta \in V^h\,, \qquad r \in [p,\infty]\,,
\label{eq:inverse} \\
| f - \pi^h f |_{k,p} & \leq C\,h^{\ell-k}\, | f |_{\ell,p} 
\qquad && \forall\ f \in W^{\ell,p}(I)\,, \label{eq:estpih} \\
| f - P^h f |_{0,p} & \leq C\,h\,| f |_{1,p} 
\qquad && \forall\ f \in W^{1,p}(I)\,. \label{eq:estph}
\end{alignat}
\end{subequations}
In order to discretize in time, let $t_m=m\,\ttau$, $m=0,\ldots,M$, 
with the uniform time step $\ttau = \frac TM >0$. 
Throughout this paper, we make use of the following short hand notations.
For a function $f \in C([0,T];B)$, with some Banach space $B$, 
we let $f^m= f(t_m)$. In addition, for a sequence of functions 
$(g^m)_{m\in\bN_0}$, $g^m \in B$, we let
\[
 D_t g^{m+1}=\frac{g^{m+1}-g^{m}}\ttau\,.
\]
For two sequences, $(f^m)_{m\in\bN_0}$ and $(g^m)_{m\in\bN_0}$, we observe
the discrete product rule
\begin{equation} \label{eq:dpr}
D_t(f^{m+1}\, g^{m+1}) = (D_t f^{m+1})\, g^{m+1} + f^m\, D_t g^{m+1}\,,
\end{equation}
as well as the following useful summation by parts formula,
for $n = 0,\ldots,M-1$:
\begin{equation} \label{eq:sbp}
\ttau\, \sum_{m=0}^n D_t f^{m+1}\, g^m 
= f^{n+1}\, g^n - \ttau\, \sum_{m=1}^n f^m\, D_t g^m, \quad 
\mbox{ if } f^0=0\,.
\end{equation}
Moreover, it is not difficult to show that for any $f\in H^1(0,T;L^2(I))$ it
holds that
\begin{equation} \label{eq:Dtf}
\ttau\,\sum_{m=k}^{n} 
|D_t f^{m+1}|_0^2 \leq \int_{t_{k}}^{t_{n+1}} |f_t|_0^2 \dt\,,
\quad 0 \leq k \leq n \leq M-1\,.
\end{equation}

Our fully discrete finite element approximation of \eqref{eq:DD} is given as
follows.

\noindent
$(\mathcal P^{h,\ttau})$
Let $\vec X^0= \pi^h \vec x(0) \in \Vhpartialzero$. For $m=0,\ldots,M-1$, 
find $\vec X^{m+1}\in \Vhpartialzero$, such that
\begin{equation} \label{eq:DDlinear}
 \left((\vec X^m\cdot\vec\ek_1)\,D_t \vec X^{m+1} , \vec\eta\,|\vec X^m_\rho|^2
\right) + \left( (\vec X^m\cdot\vec\ek_1)\,
\vec X^{m+1}_\rho,\vec\eta_\rho \right)
+ \left( \vec\eta \cdot\vec\ek_1, |\vec X^{m}_\rho|^2\right)
= 0
\quad \forall\ \vec\eta \in \Vhpartialzero\,.
\end{equation}

\begin{lem} \label{lem:ex}
Let $\vec X^m \in \Vhpartialzero$ with $\vec X^m \cdot \vec\ek_1 > 0$
and $|\vec X^m_\rho| > 0$ in $I$.
Then there exists a unique solution 
$\vec X^{m+1} \in \Vhpartialzero$ to \eqref{eq:DDlinear}. 
\end{lem}
\begin{proof}
As (\ref{eq:DDlinear}) is a linear system, where the
number of unknowns equals the number of equations,
it is enough to show uniqueness.
We hence consider the homogeneous system and assume that
$\vec X \in \Vhpartialzero$ is such that
\begin{equation} \label{eq:ex}
\left((\vec X^m\cdot\vec\ek_1)\,\vec X , \vec\eta\,|\vec X^m_\rho|^2\right)
+ \ttau \left( (\vec X^m\cdot\vec\ek_1)\,
\vec X_\rho,\vec\eta_\rho \right) = 0
\quad \forall\ \vec\eta \in \Vhpartialzero\,.
\end{equation}
Choosing $\vec\eta = \vec X$ in \eqref{eq:ex} yields that
\begin{equation*} 
\left(\vec X^m\cdot\vec\ek_1\,|\vec X|^2 , |\vec X^m_\rho|^2\right)
+ \ttau \left( \vec X^m\cdot\vec\ek_1 , |\vec X_\rho|^2 \right) = 0\,,
\end{equation*}
and our assumptions on $\vec X^m$ imply that $\vec X = \vec 0$. 
\end{proof}

We are now in a position to formulate the main result of this paper, which are optimal 
$H^1$-- and $L^2$--error bounds in the case of genus-1 surfaces.

\begin{thm} \label{thm:ebP}
Let $\partial I = \emptyset$.
Suppose that \eqref{eq:DDstrong3} has a solution $\vec x$ 
satisfying 
\begin{equation} \label{eq:regularity}
\vec x \in C([0,T];[W^{2,\infty}(I)]^2), \; \vec x_t \in L^2(0,T;[H^2(I)]^2), \; \vec x_{tt} \in L^2(0,T;[L^2(I)]^2)\,,
\end{equation}
as well as
\begin{equation} \label{eq:xrho}
|\vec x_\rho|  > 0, \quad \vec x \cdot \vec\ek_1 >0  \quad \text{in } \overline I \times [0,T]\,.
\end{equation}
Then there exist $h_0,\gamma \in \bRplus$ such that
if $0<h \leq h_0$ and $\ttau \leq \gamma\, \sqrt{h}$,
then $(\mathcal P^{h,\ttau})$ has a unique 
solution $(\vec X^{m})_{m=0,\ldots,M}$, and the following error bounds hold:
\begin{align}
\max_{m=0,\ldots,M} |  \vec x^m - \vec X^m |_0^2 
+ \ttau\, \sum_{m=0}^{M-1} |  \vec x_t^{m+1} - D_t \vec X^{m+1} |_0^2
& \leq C \left( h^4 + (\ttau)^2 \right),  \label{eq:ebP1} \\
\max_{m=0,\ldots,M} | \vec x^m - \vec X^m |_1^2 & \leq C \left( h^2 + (\ttau)^2 \right) .  \label{eq:ebP2}
\end{align}
\end{thm}

\begin{rem} 
At present we are not able to prove similar bounds for genus-0 surfaces. 
Note that in this case one has to deal in addition with the boundary 
conditions \eqref{eq:axibc} and \eqref{eq:bc}, which
state that the curve meets the $x_2$--axis at a right angle. 
An error analysis for curve shortening flow subject to normal contact to a 
given boundary has been carried out in \cite{DeckelnickE98},
based on the formulation \eqref{eq:csfdeTurck}. 
However, it is not straightforward to apply the corresponding analysis to our 
setting because of the presence of $\vec x \cdot \vec\ek_1$, which degenerates 
at the boundary. As we shall see in Section~\ref{sec:nr}, the method
nevertheless performs well in practice, also in the case of genus-0 surfaces.
\end{rem}

\setcounter{equation}{0}
\section{Error analysis} \label{sec:error}

In this section, we prove the main result of this paper, Theorem~\ref{thm:ebP}.
Hence we assume that $\partial I = \emptyset$, so that 
$I = \RZ$ and $\Vhpartialzero = \Vh$. To begin, note that  (\ref{eq:regularity}) and (\ref{eq:xrho}) imply that 
there exist constants $c_0,c_1,C_0 \in \bRplus$ such that the solution of \eqref{eq:DDstrong3} satisfies
\begin{equation} \label{eq:xvor}
\| \vec x(\cdot,t) \|_1 \leq C_0\ \mbox{ in } [0,T]\,, \quad 
c_0 \leq | \vec x_\rho | \leq C_0\ \mbox{ in } I\times [0,T]\,, \quad 
\vec x\cdot\vec\ek_1 \geq c_1\ \mbox{ in } I \times [0,T]\,.
\end{equation}
We claim that $\vec X^0,\ldots,\vec X^n$ solving $(\mathcal P^{h,\ttau})$
exist uniquely for every 
$n \in \{ 0,\ldots,M \}$ and satisfy
\begin{equation} \label{eq:indvor}
\| \vec X^m \|_1 \leq 2\, C_0\,, \quad 
\frac{c_0}2 \leq | \vec X^m_\rho | \leq 2\,C_0\ \mbox{ in } I\,, \quad 
\vec X^m\cdot\vec\ek_1 \geq \frac{c_1}2\ \mbox{ in } I\,, \; \;
m=0,\ldots,n \,,
\end{equation}
provided that $\ttau \leq \gamma\, \sqrt{h}$ for a suitably chosen $\gamma>0$. 
Clearly \eqref{eq:xvor} and \eqref{eq:estpih} imply that the 
assertion holds for $n=0$, provided that $0<h \leq h_0$ and $h_0$ is 
sufficiently small.
Now let us assume for an $n \in \{ 0,\ldots,M-1\}$ that 
$\vec X^m \in \Vh$, $0 \leq m \leq n$, 
solving \eqref{eq:DDlinear} exist and satisfy (\ref{eq:indvor}). 
Lemma~\ref{lem:ex} implies the existence of $\vec X^{n+1}$, and we shall derive
the corresponding bounds  (\ref{eq:indvor})  with the help of an error analysis. To do so, let
us decompose the error 
\begin{equation} \label{eq:errdcp}
\vec x^m - \Vec X^m= (\vec x^m - \pi^h \vec x^m) + (\pi^h \vec x^m - \vec X^m) =: \vec d^m +\vec E^m\,.
\end{equation}
We note that, for $k \in \{0,1\}$ and $p\in [2,\infty]$ we have 
from \eqref{eq:estpih}, \eqref{eq:regularity} and \eqref{eq:Dtf} that
\begin{subequations}
\begin{align}
| \vec d^m |_{k,p} & \leq C\, h^{2-k} \, | \vec x^m |_{2,p}
 \leq C\, h^{2-k}\,, \; \; m=0,\ldots,n\,, \label{eq:reg3} \\
\ttau \, \sum_{m=0}^{M-1} | D_t \vec d^{m+1} |_k^2 & 
\leq C\, h^{4-2k}\, \int_0^T | \vec x_t |_2^2 \dt \leq C\, h^{4-2k}\,.
\label{eq:reg4}
\end{align}
\end{subequations}
In addition, we infer from 
\citet[Theorem 4, Section 5.9.2]{Evans98} and (\ref{eq:regularity}) that
\[
\vec x_t \in C([0,T];[H^1(I)]^2)\,,
\]
and hence, on recalling (\ref{eq:estpih}),
\begin{equation} \label{eq:reg1}
\| D_t \vec x^{m+1} \|_1 + \| D_t \pi^h \vec x^{m+1} \|_1 
\leq C\, \| \vec x_t \|_{C([0,T];[H^1(I)]^2)} \leq C\,, 
\quad m=0,\ldots,M-1\,.
\end{equation}
Furthermore, \eqref{eq:indvor}, 
\eqref{eq:sobolev}, \eqref{eq:estpih} and \eqref{eq:regularity} imply 
that
\begin{equation} \label{eq:X1inf}
\| \vec X^m \|_{1,\infty} 
+ \| \vec E^m \|_{1,\infty} 
\leq C\,, \; \; m=0,\ldots,n\,.
\end{equation}
We begin by taking the difference of (\ref{eq:DD}) for $t=t_{m+1}$ and 
(\ref{eq:DDlinear}), in order to obtain the error relation
\begin{align*}
& \left((\vec X^m\cdot\vec\ek_1)\,
D_t \vec E^{m+1}, \vec\eta\,|\vec X^m_\rho|^2
\right) + \left( (\vec X^m\cdot\vec\ek_1)\,
\vec E^{m+1}_\rho,\vec\eta_\rho \right) \\ & \quad 
= \left[ \left( ( \vec X^m \cdot\vec\ek_1) \, D_t \pi^h \vec x^{m+1}, \vec\eta \, | \vec X^m_\rho |^2 \right) 
- \left( (\vec x^{m+1} \cdot\vec\ek_1)\, \vec x^{m+1}_t, \vec\eta \, | \vec x^{m+1}_\rho |^2 \right) \right]  \\ & \qquad
 + \left[ \left( \vec X^m \cdot\vec\ek_1) \, ( \pi^h \vec x^{m+1})_\rho, \vec \eta_\rho \right) 
- \left( ( \vec x^{m+1} \cdot\vec\ek_1 ) \, \vec x^{m+1}_\rho, \vec \eta_\rho \right) \right] 
 +  \left(  \vec\eta \cdot\vec\ek_1, \, |\vec X^m_\rho|^2 - | \vec x^{m+1}_\rho |^2 \right) \\ & \quad
=: \sum_{i=1}^3 T_i(\vec\eta) \qquad \forall\ \vec\eta\in \Vh\,.
\end{align*}
If we set $\vec\eta=\ttau\, D_t \vec E^{m+1}$ and sum over $m=0,\ldots,n$, 
we obtain
\begin{align} 
& 
\ttau \sum_{m=0}^n \left( \vec X^m\cdot\vec\ek_1\,
| D_t \vec E^{m+1} |^2, |\vec X^m_\rho|^2 \right) 
+ \sum_{m=0}^n \left( (\vec X^m\cdot\vec\ek_1)\,
\vec E^{m+1}_\rho,\vec E^{m+1}_\rho - \vec E^m_\rho \right) 
\nonumber \\ & \qquad
 = \ttau \sum_{m=0}^n \sum_{i=1}^3 T_i(D_t \vec E^{m+1})\,. \label{eq:err0}
\end{align}
Observing that 
\begin{equation} \label{eq:bba}
b\,(b-a) \geq \tfrac12\,( b^2 - a^2) \quad \forall\ a,b \in \bR\,,
\end{equation}
and using (\ref{eq:sbp}) with $\vec E^0=\vec 0$, we obtain that
\begin{align*}
& \sum_{m=0}^n \left( (\vec X^m\cdot\vec\ek_1)\,
\vec E^{m+1}_\rho,\vec E^{m+1}_\rho - \vec E^m_\rho \right) 
\geq \tfrac12 \sum_{m=0}^n \left( \vec X^m\cdot\vec\ek_1, | \vec E^{m+1}_\rho |^2 - | \vec E^m_\rho |^2 \right) \\ \nonumber & \qquad 
= \tfrac12 \left( \vec X^n\cdot\vec\ek_1, | \vec E^{n+1}_\rho |^2 \right) 
- \tfrac12\, \ttau \, 
\sum_{m=1}^n \left( D_t \vec X^m\cdot\vec\ek_1, \, | \vec E^m_\rho |^2 \right).
\end{align*}
Combining this with \eqref{eq:err0} yields, on recalling 
(\ref{eq:indvor}), that
\begin{align}
& \frac{c_0^2\,c_1}{8}\,\ttau \,\sum_{m=0}^n | D_t \vec E^{m+1} |_0^2 
+ \frac{c_1}{4}\, | \vec E^{n+1}_\rho |_0^2 \nonumber \\ & \qquad
 \leq \tfrac12\, \ttau \, \sum_{m=1}^n 
\left( D_t \vec X^m\cdot\vec\ek_1, | \vec E^m_\rho |^2 \right) 
+ \ttau \sum_{m=0}^n \sum_{i=1}^3 T_i(D_t \vec E^{m+1})\,. \label{eq:err1}
\end{align}
In order to treat the first term on the right hand side of (\ref{eq:err1}) we write $D_t \vec X^m = D_t \pi^h \vec x^m - D_t \vec E^m$.
Observing that $|\vec E^m_\rho|_{0,\infty} + | D_t \pi^h \vec x^m |_{0,\infty} 
\leq C$ in view of \eqref{eq:X1inf}, \eqref{eq:sobolev}, \eqref{eq:reg1},
we deduce that
\begin{align} 
\tfrac12\, \ttau \, \sum_{m=1}^n \left( D_t \vec X^m\cdot\vec\ek_1,  
| \vec E^m_\rho |^2 \right) 
& \leq C\,\ttau\,\sum_{m=1}^n \left( | D_t \pi^h \vec x^m |_{0,\infty}\, 
| \vec E^m_\rho|^2_0 + | \vec E^m_\rho |_{0,\infty}\, | D_t \vec E^m |_0 \,
| \vec E^m_\rho |_0 \right) \nonumber \\ 
& \leq \epsilon\,\ttau\,\sum_{m=0}^{n-1} | D_t \vec E^{m+1} |_0^2 
+ C_\epsilon\,\ttau \, \sum_{m=1}^n | \vec E^m_\rho |_0^2\,. \label{eq:err2}
\end{align}
In the above $\epsilon >0$ is a parameter that will later be chosen to
be sufficiently small, while $C_\epsilon$ is a positive constant depending
on $\epsilon$.
Next, let us write
\begin{align}
T_1(D_t \vec E^{m+1}) & = \left( ( ( \vec X^m - \vec x^{m+1}) \cdot\vec\ek_1 )\, D_t \pi^h \vec x^{m+1}, D_t \vec E^{m+1} \, | \vec X^m_\rho |^2 \right)
\nonumber \\ & \quad
+ \left( (\vec x^{m+1}\cdot\vec\ek_1) \, ( D_t \pi^h \vec x^{m+1} - \vec x_t^{m+1}), D_t \vec E^{m+1} \, | \vec X^m_\rho |^2 \right) 
\nonumber \\ & \quad
+ \left( (\vec x^{m+1}\cdot\vec\ek_1) \, \vec x_t^{m+1}, D_t \vec E^{m+1}
\left( | \vec X^m_\rho |^2 - | \vec x^{m+1}_\rho |^2 \right) \right) 
 =: \sum_{j=1}^3 T^j_1\,. \label{eq:T1_123}
\end{align}
Since $|\vec X^m_\rho |^2 \leq 4\, C_0^2$ in $I$, 
we obtain with the help of (\ref{eq:sobolev}) and (\ref{eq:reg1}) that
\begin{align}
\ttau\,\sum_{m=0}^n T^1_1 & \leq C \,\ttau\,
\sum_{m=0}^n |\vec x^{m+1} - \vec X^m|_0\,| D_t \pi^h \vec x^{m+1}|_{0,\infty}
\,| D_t \vec E^{m+1} |_0  \nonumber \\ & 
\leq C\,\ttau\,\sum_{m=0}^n |\vec x^{m+1} - \vec X^m|_0\,| D_t \vec E^{m+1}|_0
\nonumber \\ & 
\leq \epsilon\,\ttau\,\sum_{m=0}^n | D_t \vec E^{m+1} |_0^2 
+ C_\epsilon \, \ttau \, \sum_{m=0}^n | \vec x^{m+1} - \vec X^m |_0^2\,. 
\label{eq:t11a}
\end{align}
For the last term in the above relation we observe that
\begin{equation} \label{eq:Xmxm}
\vec x^{m+1} - \vec X^m = \vec E^m +  \vec d^m + \ttau \, D_t \vec x^{m+1}\,,
\end{equation}
so that we may infer from (\ref{eq:t11a}), together with (\ref{eq:reg1}) and 
(\ref{eq:reg3}), that
\begin{equation} \label{eq:t11} 
\ttau \,\sum_{m=0}^n T^1_1 \leq \epsilon\,\ttau\,\sum_{m=0}^n | D_t \vec E^{m+1} |_0^2 
+ C_{\epsilon}\,\ttau\, \sum_{m=1}^n | \vec E^m |_0^2 
+ C_\epsilon \left( h^4 + (\ttau)^2 \right) . 
\end{equation}
Let us next consider $T^2_1$. Using \eqref{eq:regularity} and again the fact 
that $|\vec X^m_\rho|^2 \leq 4\, C_0^2$ in $I$, we may estimate
\begin{align}
\ttau \, \sum_{m=0}^n T^2_1 &
\leq C\,\ttau\, \sum_{m=0}^n | D_t \pi^h \vec x^{m+1} - \vec x_t^{m+1} |_0\,
 | D_t \vec E^{m+1} |_0 \nonumber \\ & 
\leq \epsilon\,\ttau\, \sum_{m=0}^n |D_t \vec E^{m+1} |_0^2 
+ C_{\epsilon} \, \ttau \, 
\sum_{m=0}^n | D_t \pi^h \vec x^{m+1} - \vec x_t^{m+1} |_0^2 \,.
\label{eq:t21a}
\end{align}
In view of (\ref{eq:estpih}) we have
\begin{align} \label{eq:Dtest}
| D_t \pi^h \vec x^{m+1} - \vec x_t^{m+1}|_0 & 
= \left| \frac1\ttau\, \int_{t_m}^{t_{m+1}} \pi^h \vec x_t 
 - \vec x_t + \vec x_t - \vec x_t^{m+1} \dt \right|_0 
\nonumber \\ &
\leq \frac1\ttau \int_{t_m}^{t_{m+1}} | \pi^h \vec x_t - \vec x_t |_0 \dt 
+ \frac1\ttau \int_{t_m}^{t_{m+1}} | \vec x_t - \vec x^{m+1}_t|_0 \dt 
\nonumber \\ & 
\leq C\,\frac{h^2}{\sqrt\ttau} \left( \int_{t_m}^{t_{m+1}} | \vec x_t |_2^2 \dt \right)^{\frac12} + \sqrt\ttau \left( \int_{t_m}^{t_{m+1}} | \vec x_{tt} |_0^2 \dt \right)^{\frac12} ,
\end{align}
which combined with (\ref{eq:t21a}) and (\ref{eq:regularity}) implies that
\begin{equation} \label{eq:t21}
\ttau \, \sum_{m=0}^n T^2_1 \leq \epsilon\,\ttau\, \sum_{m=0}^n |D_t \vec E^{m+1} |_0^2 
+ C_{\epsilon} \left( h^4 + ( \ttau )^2 \right) .
\end{equation}
In order to treat $T^3_1$, recall \eqref{eq:T1_123}, we write,
on noting \eqref{eq:Xmxm},
\begin{align} 
& | \vec x^{m+1}_\rho |^2 - | \vec X^m_\rho |^2 =  
2 \,( \vec x^{m+1}_\rho- \vec X^m_\rho )\cdot\vec x^{m+1}_\rho 
- | \vec x^{m+1}_\rho - \vec X^m_\rho |^2 \nonumber \\ & \qquad 
= 2\,\vec E^m_\rho\cdot\vec x^{m+1}_\rho 
+ 2\,\vec d^m_\rho\cdot\vec x^{m+1}_\rho 
+ 2\,\ttau \, D_t \vec x^{m+1}_\rho\cdot\vec x^{m+1}_\rho 
- | \vec E^m_\rho + \vec d^m_\rho + \ttau \, D_t \vec x^{m+1}_\rho |^2 
\nonumber \\ & \qquad
=: 2\,\vec d^m_\rho\cdot\vec x^{m+1}_\rho + u^m \,,
\label{eq:xXnormdiff}
\end{align}
where, in view of (\ref{eq:regularity}),  (\ref{eq:reg1}), (\ref{eq:reg3})
and \eqref{eq:X1inf},
\begin{align}
| u^m |_0 & \leq C \left( | \vec E^m_\rho |_0 + 
\ttau \, | D_t \vec x^{m+1}_\rho |_0 + | \vec E^m_\rho |_{0,\infty}\, 
| \vec E^m_\rho |_0 + | \vec d^m_\rho |_{0,\infty}\,| \vec d^m_\rho |_0 
\right) \nonumber \\ & \quad 
+ C \, \ttau \, | \vec x^{m+1}_\rho - \vec x^m_\rho |_{0,\infty}\, 
| D_t \vec x^{m+1}_\rho|_0  \nonumber \\ & 
\leq C \left( | \vec E^m_\rho |_0  + h^2 +  \ttau \right) . \label{eq:um}
\end{align}
We obtain with the help of (\ref{eq:xXnormdiff}) and integration by parts that
\begin{align}
\ttau \,\sum_{m=0}^n T^3_1 & = -\ttau \,\sum_{m=0}^n 
\left( (\vec x^{m+1}\cdot\vec\ek_1) \, \vec x_t^{m+1}\cdot D_t \vec E^{m+1}, 
2\,\vec d^m_\rho\cdot\vec x^{m+1}_\rho + u^m  \right) 
\nonumber \\ & 
= 2\,\ttau\,\sum_{m=0}^n \left(  (\vec x^{m+1}\cdot\vec\ek_1) \, 
(\vec x_t^{m+1}\cdot D_t \vec E^{m+1}_\rho)\, \vec x^{m+1}_\rho, 
\vec d^m \right) \nonumber \\ & \quad
+ 2 \, \ttau \, \sum_{m=0}^n \left( \vec v^m, \vec d^m \right)
- \ttau \, \sum_{m=0}^n \left( (\vec x^{m+1}\cdot\vec\ek_1) \, 
\vec x_t^{m+1} \cdot D_t \vec E^{m+1}, u^m \right) ,  \label{eq:t31a}
\end{align}
where
\begin{align*}
\vec v^m & = (\vec x^{m+1}_\rho \cdot\vec\ek_1) \, (\vec x_t^{m+1}\cdot D_t \vec E^{m+1}) \, \vec x^{m+1}_\rho + (\vec x^{m+1}\cdot\vec\ek_1) \, 
(\vec x_{t,\rho}^{m+1}\cdot D_t \vec E^{m+1})\, \vec x^{m+1}_\rho \\ & \quad 
+ (\vec x^{m+1}\cdot\vec\ek_1) \, 
(\vec x_t^{m+1}\cdot D_t \vec E^{m+1})\, \vec x^{m+1}_{\rho \rho}\,.
\end{align*}
Using H\"older's inequality, (\ref{eq:sobolev}) and the fact that
$\| \vec x^{m+1} \|_{2,\infty} + \| \vec x_t^{m+1} \|_1 \leq C$, 
we have
\[
| \vec v^m |_{0,1} 
\leq C \, \| \vec x^{m+1}_t \|_1 \, | D_t \vec E^{m+1} |_0 
\leq C \, | D_t \vec E^{m+1} |_0\,.
\]
If we combine this bound with (\ref{eq:um}), and use 
\eqref{eq:regularity} and (\ref{eq:reg3}),
we infer that
\begin{align}
& 2 \, \ttau \, \sum_{m=0}^n ( \vec v^m, \vec d^m) 
- \ttau \, \sum_{m=0}^n \left( (\vec x^{m+1}\cdot\vec\ek_1) \, 
\vec x_t^{m+1} \cdot D_t \vec E^{m+1}, u^m  \right) \nonumber \\ & \qquad 
\leq C \, \ttau \, \sum_{m=0}^n | \vec v^m |_{0,1}\, | \vec d^m |_{0,\infty} 
+ C \, \ttau \, \sum_{m=0}^n | D_t \vec E^{m+1} |_0 \, | u^m |_0 
\nonumber \\ & \qquad
\leq C \,\ttau\,\sum_{m=0}^n \left( | \vec E^m_\rho |_0 + h^2 + \ttau \right) 
| D_t \vec E^{m+1} |_0 \nonumber \\ & \qquad 
\leq \epsilon \, \ttau \, \sum_{m=0}^n | D_t \vec E^{m+1} |_0^2 +
C_\epsilon \, \ttau \, \sum_{m=1}^n | \vec E^m_\rho |_0^2 
+ C_\epsilon \left( h^4 + (\Delta t)^2 \right). \label{eq:t31b}
\end{align}
It remains to consider 
the first term on the right hand side of \eqref{eq:t31a}. We
deduce from (\ref{eq:sbp}) with $\vec E^0 = \vec 0$, \eqref{eq:regularity},
\eqref{eq:dpr}, (\ref{eq:reg3}), (\ref{eq:reg4}), \eqref{eq:Dtf} 
and \eqref{eq:reg1} that
\begin{align*}
& \ttau\,\sum_{m=0}^n \left( (\vec x^{m+1}\cdot\vec\ek_1)\, 
(\vec x_t^{m+1}\cdot D_t \vec E^{m+1}_\rho)\, \vec x^{m+1}_\rho, 
\vec d^m \right) \nonumber \\ & \quad
= \left( (\vec x^{n+1}\cdot\vec\ek_1)\, 
(\vec x_t^{n+1}\cdot\vec E^{n+1}_\rho)\,\vec x^{n+1}_\rho, \vec d^n \right) 
-\ttau\, \sum_{m=1}^n \left( D_t [ (\vec x^{m+1}\cdot\vec\ek_1)\, 
(\vec x^{m+1}_\rho\cdot\vec d^m )\, \vec x^{m+1}_t], \vec E^m_\rho \right) 
\\ &\quad
\leq C\,| \vec E^{n+1}_\rho |_0 \, | \vec d^n |_0 
+ C\,\ttau\,\sum_{m=1}^n  
\left( | \vec d^m |_{0,\infty} \left[ | D_t \vec x^{m+1}_t |_0 + \| D_t \vec x^{m+1} \|_1 \right] 
+ | D_t \vec d^m |_0\, | \vec x^m_t |_{0,\infty} \right) 
| \vec E^m_\rho |_0 \\ & \quad
\leq \epsilon\,| \vec E^{n+1}_\rho |_0^2 
+\ttau\, \sum_{m=1}^n | \vec E^m_\rho |_0^2 + C_\epsilon\, h^4\,.
\end{align*}
This implies together with (\ref{eq:t31a}) and (\ref{eq:t31b}) that
\begin{equation}  \label{eq:t31}
\ttau \,\sum_{m=0}^n T^3_1 \leq  \epsilon\,| \vec E^{n+1}_\rho |_0^2 + \epsilon\,\ttau\, \sum_{m=0}^n | D_t \vec E^{m+1} |_0^2
+ C_\epsilon \, \ttau\, \sum_{m=1}^n | \vec E^m_\rho |_0^2 + C_{\epsilon} \left( h^4 + (\ttau)^2 \right) .
\end{equation}
Combining \eqref{eq:T1_123}, (\ref{eq:t11}), (\ref{eq:t21}) and (\ref{eq:t31}) we obtain that
\begin{equation} \label{eq:t1}
\ttau\,\sum_{m=0}^n T_1(D_t \vec E^{m+1}) 
\leq \epsilon \left(| \vec E^{n+1}_\rho |_0^2 
+ \ttau \, \sum_{m=0}^n | D_t \vec E^{m+1} |_0^2 \right) 
+ C_\epsilon \left( \ttau \, \sum_{m=1}^n \| \vec E^m \|_1^2 
+ h^4 + (\ttau)^2 \right) .
\end{equation}
Let us next investigate the terms involving 
$T_2(D_t \vec E^{m+1})$ in \eqref{eq:err1}.
Since $\int_{I_j} (f- \pi^h f)_\rho\,\eta_\rho \drho=0$ for all 
$\eta \in V^h$ and $j=1,\ldots,J$, and on recalling the definitions of 
$P^h$ and $\vec d^m$ from \eqref{eq:defph} and \eqref{eq:errdcp}, we may write
\begin{align} 
& T_2(D_t \vec E^{m+1})
= \left( (\vec X^m \cdot\vec\ek_1) \, ( \pi^h \vec x^{m+1})_\rho, 
D_t \vec E^{m+1}_\rho \right) 
- \left( ( \vec x^{m+1} \cdot\vec\ek_1 ) \, \vec x^{m+1}_\rho, 
D_t \vec E^{m+1}_\rho \right)  \nonumber \\ & \qquad
 =  \left( ( (\vec X^m - \vec x^{m+1})\cdot\vec\ek_1 )\, \vec x^{m+1}_\rho, D_t \vec E^{m+1}_\rho \right)  
- \left( ( \vec X^m\cdot\vec\ek_1 - P^h[\vec X^m\cdot\vec\ek_1] )\, \vec d^{m+1}_\rho, D_t \vec E^{m+1}_\rho \right) \nonumber \\ & \qquad
 = : T^1_2 + T^2_2\,. \label{eq:T2_12} 
\end{align}
Applying (\ref{eq:sbp}) with $\vec E^0 = \vec 0$, we obtain that
\begin{align} \label{eq:t12a} 
\ttau \, \sum_{m=0}^n T^1_2 &
= \left( ( (\vec X^n - \vec x^{n+1})\cdot\vec\ek_1 )\, \vec x^{n+1}_\rho, 
\vec E^{n+1}_\rho \right) 
- \ttau\,\sum_{m=1}^n \left( D_t \left[ ( 
(\vec X^m - \vec x^{m+1}) \cdot\vec\ek_1 ) \,
\vec x^{m+1}_\rho \right], \vec E^m_\rho \right) 
\nonumber  \\ & =: \widetilde T - \ttau\, \sum_{m=1}^n \widetilde T^1_2\,. 
\end{align}
Using the identity
\begin{equation*} 
\vec x^{n+1} - \vec X^n = 
\vec d^{n+1} + \vec E^{n+1} + \ttau\,D_t \pi^h \vec x^{n+1}
- \ttau\,D_t\vec E^{n+1}\,,
\end{equation*}
and on recalling \eqref{eq:regularity}, \eqref{eq:indvor}, \eqref{eq:sobolev},
\eqref{eq:reg3} and \eqref{eq:reg1} we may estimate
\begin{align*}
\widetilde T & \leq C\,| \vec x^{n+1} - \vec X^n |_0 \, | \vec E^{n+1}_\rho |_0  \nonumber \\
& \leq C \left( | \vec d^{n+1} |_0 + | \vec E^{n+1} |_0  
+ \ttau\, | D_t \pi^h \vec x^{n+1} |_0 + \ttau\,| D_t \vec E^{n+1}|_0 \right) 
| \vec E^{n+1}_\rho |_0 \nonumber \\ 
& \leq \epsilon \left( | \vec E^{n+1}_\rho |_0^2 
+ \ttau\, | D_t \vec E^{n+1} |_0^2 \right) 
+ C_\epsilon\,\ttau\, | \vec E^{n+1}_\rho |_0^2 
+ C_\epsilon\, | \vec E^{n+1} |_0^2 
+ C_\epsilon \left( h^4 + (\ttau)^2 \right) .
\end{align*}
Regarding the second term on the right hand side of \eqref{eq:t12a},
we obtain, on noting \eqref{eq:dpr}, \eqref{eq:regularity},  
\eqref{eq:Xmxm}, \eqref{eq:sobolev}, 
(\ref{eq:reg1}), (\ref{eq:reg3}), \eqref{eq:Dtf}  
and (\ref{eq:reg4}), that
\begin{align*} 
 - \ttau\, \sum_{m=1}^n \widetilde T^1_2 
& \leq C\, \ttau\,\sum_{m=1}^n \left[ |  \x^{m+1} - \vec X^m  |_{0,\infty}\, 
| D_t \vec x^{m+1}_\rho |_0 +  | D_t (\vec x^{m+1} - \vec X^m)|_0 \right] | \vec E^m_\rho |_0 
\nonumber \\ & 
\leq C \, \ttau\, \sum_{m=1}^n \left( \| \vec E^m \|_1 + | \vec d^m |_{0,\infty} + \ttau\,\| D_t \x^{m+1} \|_1 
\right) | D_t \vec x^{m+1}_\rho |_0\,| \vec E^m_\rho |_0 \\ & \quad 
+ C\,\ttau\, \sum_{m=1}^n \left(  
| D_t \vec E^m |_0 + | D_t \vec d^m |_0 
+ \ttau\, | D_t D_t \vec x^{m+1} |_0  \right) 
| \vec E^m_\rho |_0 \nonumber \\
& \leq \epsilon \, \ttau\,\sum_{m=0}^{n-1} | D_t \vec E^{m+1} |_0^2  
+ C_\epsilon\,\ttau\, \sum_{m=1}^n \| \vec E^m \|_1^2 
+ C_\epsilon \left( h^4 + (\ttau)^2 \right),
\end{align*}
where we used the fact that 
$\sum_{m=1}^n \ttau\,| D_t D_t \vec x^{m+1} |_0^2 \leq C\, 
\int_0^T | \vec x_{tt} |_0^2 \dt \leq C$. 
If we insert the above estimates into (\ref{eq:t12a}) we deduce that 
\begin{align}
\ttau \, \sum_{m=0}^n T^1_2 & 
\leq  \left( \epsilon + C_\epsilon\, \ttau \right)  | \vec E^{n+1}_\rho |_0^2 
+ \epsilon \,  \ttau\,\sum_{m=0}^n | D_t \vec E^{m+1} |_0^2 
+ C_\epsilon\,| \vec E^{n+1} |_0^2  \nonumber \\ & \quad
+ C_\epsilon\,\ttau\, \sum_{m=1}^n \| \vec E^m \|_1^2 
+ C_\epsilon \left( h^4 + (\ttau)^2 \right) . \label{eq:t12}
\end{align}
A further application of (\ref{eq:sbp}) yields together with 
\eqref{eq:dpr}, (\ref{eq:estph}) and 
(\ref{eq:reg3}), on recalling \eqref{eq:X1inf}, that
\begin{align} \label{eq:t22a}
& \ttau \, \sum_{m=0}^n T^2_2  
= - \left( ( \vec X^n\cdot\vec\ek_1
- P^h[\vec X^n\cdot\vec\ek_1] )\, \vec d^{n+1}_\rho, 
\vec E^{n+1}_\rho \right)  \nonumber \\ & \qquad \qquad \qquad
+\ttau\, \sum_{m=1}^n \left( 
D_t \left[ ( \vec X^m\cdot\vec\ek_1 - P^h[\vec X^m\cdot\vec\ek_1] 
)\, \vec d^{m+1}_\rho \right], \vec E^m_\rho \right) \nonumber \\ &
\leq C\, h\, | ( \vec X^n\cdot\vec\ek_1)_\rho |_{0,\infty}\,
 | \vec d^{n+1}_\rho |_0 \, | \vec E^{n+1}_\rho |_0 
+ C\,\ttau\,h\,\sum_{m=1}^n | ( \vec X^m\cdot\vec\ek_1)_\rho |_{0,\infty} 
\,| D_t \vec d^{m+1}_\rho |_0 \, | \vec E^m_\rho |_0 \nonumber \\ & \quad
+ C\,\ttau\,h\, \sum_{m=1}^n | D_t ( \vec X^m\cdot\vec\ek_1)_\rho |_0 \,
| \vec d^m_\rho |_{0,\infty} \, | \vec E^m_\rho |_0 \nonumber \\ &
\leq C\, h^2 \, | \vec E^{n+1}_\rho |_0 
+ C\,\ttau\, \sum_{m=1}^n | \vec E^m_\rho |_0^2 
+ C\,h^2\,\ttau\, \sum_{m=1}^n | D_t \vec d^{m+1} |_1^2 
+ C\,\ttau\,h^2\, \sum_{m=1}^n \| D_t \vec X^m \|_1 \,| \vec E^m_\rho |_0\,. 
\end{align}
We have that $\| D_t \vec X^m \|_1 \leq \| D_t \vec E^m \|_1 
+ \| D_t  \pi^h \vec x^m \|_1 \leq C\,(h^{-1}\, | D_t \vec E^m |_0 +1)$,
on recalling \eqref{eq:inverse} and (\ref{eq:reg1}).
Hence it follows from \eqref{eq:t22a} and (\ref{eq:reg4}) that
\begin{align}
\ttau \, \sum_{m=0}^n T^2_2 & 
\leq C\, h^2\, | \vec E^{n+1}_\rho |_0 
+ C\,\ttau\, \sum_{m=1}^n | \vec E^m_\rho |_0^2 
+ C\, h^4 
+ C\, h\,\ttau\, \sum_{m=1}^n | D_t \vec E^m |_0\, | \vec E^m_\rho |_0 
\nonumber \\ & 
\leq \epsilon \left( | \vec E^{n+1}_\rho |_0^2 + 
\ttau\,\sum_{m=0}^{n-1} | D_t \vec E^{m+1} |_0^2 \right) 
+ C_\epsilon\,\ttau\, \sum_{m=1}^n | \vec E^m_\rho |_0^2 
+ C_\epsilon\, h^4\,. \label{eq:t22}
\end{align}
If we combine (\ref{eq:t12}) and (\ref{eq:t22}) we obtain, on recalling
\eqref{eq:T2_12}, that 
\begin{align} \label{eq:t2} 
\ttau\,\sum_{m=0}^n T_2(D_t \vec E^{m+1}) & 
\leq \left( \epsilon + C_\epsilon\,\ttau \right)  | \vec E^{n+1}_\rho |_0^2 
+ \epsilon \, \ttau\,\sum_{m=0}^n | D_t \vec E^{m+1} |_0^2  
+ C_\epsilon\,| \vec E^{n+1} |_0^2 \nonumber \\ & \quad
+ C_\epsilon\,\ttau\,\sum_{m=1}^n \| \vec E^m \|_1^2  
+ C_\epsilon \left( h^4 + (\ttau)^2 \right) .
\end{align}
The term $T_3(D_t \vec E^{m+1})$ can be treated in a similar way as $T^3_1$, 
and we obtain
\begin{equation} \label{eq:t3} 
\ttau \, \sum_{m=0}^n T_3(D_t \vec E^{m+1}) \leq \epsilon \,  | \vec E^{n+1}_\rho |_0^2 + \epsilon \, 
\ttau\,\sum_{m=0}^n | D_t \vec E^{m+1} |_0^2 +  C_\epsilon\,\ttau\, \sum_{m=1}^n | \vec E^m_\rho |_0^2 
+ C_\epsilon \left( h^4 + (\ttau)^2 \right) . 
\end{equation}
Let us insert (\ref{eq:err2}), (\ref{eq:t1}), (\ref{eq:t2}) and (\ref{eq:t3}) 
into (\ref{eq:err1}). After first choosing $\epsilon>0$ and then $\ttau>0$ 
sufficiently small, we obtain
\begin{equation} \label{eq:err3}
\frac{c_0^2\, c_1}{16}\,\ttau\, \sum_{m=0}^n | D_t \vec E^{m+1} |_0^2 
+ \frac{c_1}{8}\, | \vec E^{n+1}_\rho |_0^2 
\leq C\,| \vec E^{n+1} |_0^2 + C\,\ttau\, \sum_{m=1}^n \| \vec E^m \|_1^2 
+ C \left( h^4 + (\ttau)^2 \right) . 
\end{equation}
Finally, recalling that $\vec E^0=\vec 0$ we may write
\[
| \vec E^{n+1} |_0^2 
= \sum_{m=0}^n \left( | \vec E^{m+1} |_0^2 - | \vec E^m |_0^2 \right) 
= \ttau \,\sum_{m=0}^n \left( D_t \vec E^{m+1}, \vec E^m+ \vec E^{m+1} \right),
\]
and hence
\[
2\, C \, | \vec E^{n+1} |_0^2 
\leq \epsilon\,\ttau\, \sum_{m=0}^n | D_t \vec E^{m+1} |_0^2 
+ C_\epsilon\,\ttau\, \sum_{m=1}^n | \vec E^m |_0^2 
+ C_\epsilon\,\ttau\, | \vec E^{n+1} |_0^2\,.
\]
Adding this bound to (\ref{eq:err3}) and choosing first $\epsilon>0$ 
and then $\ttau>0$ sufficiently small, we deduce that
\begin{equation} \label{eq:before4}
\ttau\,\sum_{m=0}^n | D_t \vec E^{m+1} |_0^2 + \| \vec E^{n+1} \|_1^2 
\leq C\,\ttau\, \sum_{m=1}^n \| \vec E^m \|_1^2 
+ C \left( h^4 + (\ttau)^2 \right).
\end{equation}
The same arguments as above show that (\ref{eq:before4}) also holds 
with $n$ replaced by $k$, for all $0 \leq k \leq n$. 
Hence the discrete Gronwall inequality implies that 
\begin{equation} \label{eq:err4}
\ttau\,\sum_{m=0}^n | D_t \vec E^{m+1} |_0^2 + \| \vec E^{n+1} \|_1^2  
 \leq C \left( h^4 + (\ttau)^2 \right).
\end{equation}
Let us use (\ref{eq:err4}) in order to show that (\ref{eq:indvor}) holds for 
$\vec X^{n+1}$. Clearly,
we have from \eqref{eq:inverse}, \eqref{eq:reg3},
\eqref{eq:err4} and $\ttau \leq \gamma\,\sqrt{h}$ that
\begin{align} \label{eq:err5}
| \vec X^{n+1}_\rho - \vec x^{n+1}_\rho |_{0,\infty} &
\leq | \vec E^{n+1}_\rho |_{0,\infty} + | \vec d^{n+1}_\rho |_{0,\infty} 
\nonumber \\ &
\leq C\,h^{-\frac12}\,| \vec E^{n+1}_\rho |_0 
+ C\, h 
\leq C \left( \ttau\, h^{-\frac12} + h\right) \leq C\,(\gamma + h) \,.
\end{align}
Combining \eqref{eq:err5} and \eqref{eq:xvor} yields 
that $\frac{c_0}2 \leq | \vec X^{n+1}_\rho | \leq 2\, C_0$ 
provided that $0 < h \leq h_0$ and $h_0, \gamma$ are small enough. 
The remaining bounds in (\ref{eq:indvor}) can be shown in a similar way,
thus completing the induction step, 
so that the discrete solution $(\vec X^{m})_{m=0,\ldots,M}$ exists 
and satisfies \eqref{eq:indvor} for $m=0,\ldots,M$. Furthermore,
the above error analysis yields that
\begin{equation} \label{eq:err6}
\ttau\,\sum_{m=0}^{M-1} | D_t \vec E^{m+1} |_0^2 + \max_{m=0,\ldots,M} \| \vec E^m \|_1^2  
 \leq C \left( h^4 + (\ttau)^2 \right).
\end{equation}
The bounds for $\max_{m=0,\ldots,M} |\vec x^m - \vec X^m |_0^2$ and 
$\max_{m=0,\ldots,M} |\vec x^m - \vec X^m |_1^2$ 
in \eqref{eq:ebP1} and \eqref{eq:ebP2} now follow from (\ref{eq:err6}), 
\eqref{eq:errdcp} and \eqref{eq:reg3}. 
Finally, (\ref{eq:err6}) together with \eqref{eq:errdcp} and (\ref{eq:Dtest}) 
implies that
\[
\ttau\, \sum_{m=0}^{M-1} | \vec x_t^{m+1} - D_t \vec X^{m+1} |_0^2 
\leq 2\, \ttau \, \sum_{m=0}^{M-1} 
\left( | \vec x_t^{m+1} - D_t \pi^h \vec x^{m+1} |_0^2 
+ | D_t \vec E^{m+1} |_0^2 \right) 
\leq C \left( h^4 + (\ttau)^2 \right),
\]
completing the proof of Theorem \ref{thm:ebP}.

\begin{rem} 
Note that \eqref{eq:err6} yields a superconvergence result for the 
$H^1$--seminorm of the error, in that 
\[
\max_{m=0,\ldots,M} | \pi^h \vec x^m - \vec X^m |_1 \leq C\, h^2,
\]
provided that $\ttau \leq C \, h^2$\,.
\end{rem}

\setcounter{equation}{0}
\section{An alternative formulation} \label{sec:alt}

In this section we briefly consider an alternative formulation
of axisymmetric mean curvature flow in the case $I=\RZ$. 
By way of motivation, let us briefly return to \eqref{eq:DD}. 
In order to derive an apriori bound for the solution, a natural idea, 
that can be mimicked at the discrete level, is to choose 
$\vec\eta = \vec x_t$ as a test function. This yields
\[
 \left( \vec x \cdot\vec\ek_1\,|\vec x_t|^2,|\vec x_\rho|^2\right)
+ \left( (\vec x\cdot\vec\ek_1)\,\vec x_\rho, (\vec x_t)_\rho \right) 
+ \left( \vec x_t\cdot\vec\ek_1,|\vec x_\rho|^2 \right)
= 0\,,
\]
and hence
\[
\left( \vec x \cdot\vec\ek_1\,|\vec x_t|^2,|\vec x_\rho|^2\right) + 
\tfrac12\,
\ddt \left( \vec x \cdot\vec\ek_1,|\vec x_\rho|^2 \right) =
- \tfrac12\left( \vec x_t \cdot\vec\ek_1,|\vec x_\rho|^2\right).
\]
However, it does not seem possible to control the term on the right hand side,
unless a lower bound of the form $\vec x\cdot\vec\ek_1 \geq c_1 > 0$ in 
$I \times [0,T]$ is available for the solution. 
It is therefore unlikely that we can prove an unconditional stability bound
for a fully discrete approximation of \eqref{eq:DD}.

However, the situation changes if we consider the equation that results from 
multiplying \eqref{eq:DDstrong3} by $\vec x\cdot\vec\ek_1$, i.e.\
\begin{equation} \label{eq:DDstrongmult}
(\vec x\cdot\vec\ek_1)^2 \,|\vec x_\rho|^2\,\vec x_t - (\vec x\cdot\vec\ek_1)
((\vec x\cdot\vec\ek_1)\,\vec x_{\rho})_\rho + (\vec x\cdot\vec\ek_1)
|\vec x_\rho|^2\,\vec\ek_1 = \vec 0\,.
\end{equation}
As mentioned in the introduction, changing the tangential component of 
$\vec x_t$ just amounts to a reparametrization of 
$\Gamma(t) = \vec x(\overline I,t)$. It does not affect the evolution of
$\Gamma(t)$. We utilize this fact by adding the tangential term
\[
- (\vec x\cdot\vec\ek_1)\,(\vec x_\rho\cdot\vec\ek_1)\,\vec x_\rho
\]
to (\ref{eq:DDstrongmult}), which allows us to write the second order term 
in derivative form, namely
\begin{equation} \label{eq:DD2strong}
(\vec x\cdot\vec\ek_1)^2\,|\vec x_\rho|^2\,\vec x_t - 
((\vec x\cdot\vec\ek_1)^2\,\vec x_{\rho})_\rho +
(\vec x\cdot\vec\ek_1)\,|\vec x_\rho|^2\,\vec\ek_1 = \vec 0\,.
\end{equation}

A weak formulation for (\ref{eq:DD2strong}) then is given by

\noindent
$(\mathcal Q)$
Let $\vec x(0) \in [H^1(I)]^2$. For $t \in (0,T]$
find $\vec x(t) \in [H^1(I)]^2$ such that
\begin{equation} \label{eq:DD2}
 \left( (\vec x \cdot\vec\ek_1)^2\,\vec x_t, \vec\eta\,|\vec x_\rho|^2
\right)
+ \left( (\vec x\cdot\vec\ek_1)^2 \,\vec x_\rho, \vec\eta_\rho \right)
+ \left( \vec x \cdot\vec\ek_1,\vec\eta\cdot\vec\ek_1\,|\vec x_\rho|^2 \right)
 = 0
\qquad \forall\ \vec\eta \in [H^1(I)]^2 \,,
\end{equation}
while our fully discrete approximation of (\ref{eq:DD2}) reads:

\noindent
$(\mathcal Q^{h,\ttau})$
Let $\vec X^0= \pi^h \vec x(0) \in \Vh$. For $m=0,\ldots,M-1$, 
find $\vec X^{m+1}\in \Vh$, such that
\begin{align}
& \left((\vec X^m\cdot\vec\ek_1)^2\,
D_t \vec X^{m+1}, \vec\eta\,|\vec X^m_\rho|^2
\right)
+ \left( (\vec X^m\cdot\vec\ek_1)^2\,
\vec X^{m+1}_\rho,\vec\eta_\rho \right)
+ \left( \vec X^{m+1}\cdot\vec\ek_1,
\vec\eta\cdot \vec\ek_1 \, |\vec X^{m+1}_\rho|^2 \right)
= 0
\nonumber \\ & \hspace{11cm}
\quad \forall\ \vec\eta \in \Vh\,.
\label{eq:DD2nonlinear}
\end{align}
We remark that while our scheme $(\mathcal P^{h,\ttau})$, recall
\eqref{eq:DDlinear}, requires the solution of two independent linear systems at
each time step, for the approximation $(\mathcal Q^{h,\ttau})$ one has to solve
a coupled nonlinear system at each time step. \\
If we now choose $\vec\eta = \vec x_t$ in (\ref{eq:DD2}), we obtain that
\[
\tfrac12\,
\ddt \left( (\vec x \cdot\vec\ek_1)^2,|\vec x_\rho|^2 \right)
 = \left( (\vec x\cdot\vec\ek_1)^2\,\vec x_\rho, (\vec x_t)_\rho \right)
+ \left( \vec x \cdot\vec\ek_1,\vec x_t \cdot\vec\ek_1\,|\vec x_\rho|^2\right)
= - \left( (\vec x \cdot\vec\ek_1)^2\,|\vec x_t|^2,|\vec x_\rho|^2\right)
\leq 0\,, 
\]
from which we immediately obtain an apriori estimate for the solution.
Mimicking the same testing procedure on the discrete level yields the
following unconditional stability result.
\begin{lem} 
Let $(\vec X^{m})_{m=0,\ldots,M}$ be a solution to 
$(\mathcal{Q}^{h,\ttau})$. Then it holds that
\begin{equation} \label{eq:stab}
\tfrac12 \left( (\vec X^{m+1}\cdot\vec\ek_1)^2, 
|\vec X^{m+1}_\rho|^2 \right)
+ \ttau 
\left( (\vec X^m\cdot\vec\ek_1)^2\, | D_t \vec X^{m+1} |^2, |\vec X^m_\rho|^2
\right) 
\leq \tfrac12 \left( (\vec X^m\cdot\vec\ek_1)^2,|\vec X^m_\rho|^2 \right) ,
\end{equation}
for $m=0,\ldots,M-1$.
In particular, for $n\in\{0,\ldots,M-1\}$ we have that
\begin{equation} \label{eq:stabstab}
\tfrac12 \left( (\vec X^{n+1}\cdot\vec\ek_1)^2, 
|\vec X^{n+1}_\rho|^2 \right)
+ \ttau\,\sum_{m=0}^n
\left( (\vec X^m\cdot\vec\ek_1)^2\, | D_t \vec X^{m+1} |^2, |\vec X^m_\rho|^2
\right) 
\leq \tfrac12 \left( (\vec X^0\cdot\vec\ek_1)^2,|\vec X^0_\rho|^2 \right) .
\end{equation}
\end{lem}
\begin{proof}
Choosing $\vec\eta = \ttau\, D_t \vec X^{m+1}$ in (\ref{eq:DD2nonlinear})  and recalling (\ref{eq:bba})
yields that
\begin{align*} 
0 & = \ttau \left( (\vec X^m\cdot\vec\ek_1)^2\, | D_t \vec X^{m+1} |^2, |\vec X^m_\rho|^2
\right) + \left( (\vec X^m\cdot\vec\ek_1)^2\,
\vec X^{m+1}_\rho, (\vec X^{m+1} - \vec X^m)_\rho
\right) \nonumber \\
& \quad + \left( \vec X^{m+1}\cdot\vec\ek_1,
(\vec X^{m+1} - \vec X^m)\cdot\vec\ek_1 \, |\vec X^{m+1}_\rho|^2 \right)
\nonumber \\ 
& \geq \ttau \left( (\vec X^m\cdot\vec\ek_1)^2\, | D_t \vec X^{m+1} |^2, |\vec X^m_\rho|^2
\right) + \tfrac12 \left( (\vec X^m\cdot\vec\ek_1)^2,
|\vec X^{m+1}_\rho|^2 - |\vec X^m_\rho|^2 \right) \nonumber \\
& \quad
+ \tfrac12 \left( (\vec X^{m+1}\cdot\vec\ek_1)^2 - 
(\vec X^m\cdot\vec\ek_1)^2, 
|\vec X^{m+1}_\rho|^2 \right)
\nonumber \\ 
& = 
\ttau \left( (\vec X^m\cdot\vec\ek_1)^2\, | D_t \vec X^{m+1} |^2, |\vec X^m_\rho|^2
\right) + \tfrac12 \left( (\vec X^{m+1}\cdot\vec\ek_1)^2, 
|\vec X^{m+1}_\rho|^2 \right)
 - \tfrac12 \left( (\vec X^m\cdot\vec\ek_1)^2,|\vec X^m_\rho|^2 \right) .
\end{align*} 
This proves \eqref{eq:stab}. Summing for
$m=0,\ldots,n$ then yields the desired result \eqref{eq:stabstab}. 
\end{proof}

We conclude this section by remarking that with the help of a Brouwer fixed
point theorem, see e.g.\ \citet[Prop.~2.8]{Zeidler86}, it is possible to
prove the existence of a solution $\vec X^{m+1} \in \Vh$ to
\eqref{eq:DD2nonlinear}, provided that $\ttau$ is chosen
sufficiently small. The uniqueness of this solution can also be established.
Finally, using the techniques from Section~\ref{sec:error}, we can prove
that the solutions of $(\mathcal Q^{h,\ttau})$ satisfy error bounds
similar to the ones in Theorem~\ref{thm:ebP}.

\setcounter{equation}{0}
\section{Numerical results} \label{sec:nr}

\subsection{Genus-1 surfaces} \label{subsec:torus}

In order to perform a convergence experiment for an evolving torus, 
we compute a right hand side $\vec f_{\mathcal P}$ 
for \eqref{eq:DDstrong3} so that
\begin{equation} \label{eq:solx}
\vec x(\rho, t) = 
\begin{pmatrix}
g(t) + \cos(2\,\pi\,\rho) \\ \sin(2\,\pi\,\rho)
\end{pmatrix},
\quad \text{where } g(t) = 2 + \sin(\pi\,t)\,,
\end{equation}
is the solution. We then compare e.g.\ $\vec x^{m+1}$ to the discrete solution
$\vec X^{m+1}$ of \eqref{eq:DDlinear}, with the added right hand side
$(\pi^h\vec f^{m+1}_{\mathcal P}, \vec\eta)$. 
As the time step size we choose $\ttau = h^2$,
for $h = J^{-1} = 2^{-k}$, $k=5,\ldots,9$. The results in 
Table~\ref{tab:torusPquad} confirm the optimal convergence rate from
Theorem~\ref{thm:ebP}.
\begin{table}
\center
\begin{tabular}{|r|c|c|c|c|}
\hline
$J$ & $\displaystyle\max_{m=0,\ldots,M} |\vec x^m - \vec X^m|_0$ & EOC
& $\displaystyle\max_{m=0,\ldots,M} |\vec x^m - \vec X^m|_1$ & EOC 
\\ \hline
32  & 7.8742e-03 & ---  & 3.5678e-01 & ---  \\
64  & 1.9647e-03 & 2.00 & 1.7815e-01 & 1.00 \\
128 & 4.9092e-04 & 2.00 & 8.9045e-02 & 1.00 \\
256 & 1.2272e-04 & 2.00 & 4.4519e-02 & 1.00 \\
512 & 3.0678e-05 & 2.00 & 2.2259e-02 & 1.00 \\
\hline
\end{tabular}
\caption{Errors for the convergence test for (\ref{eq:solx})
over the time interval $[0,1]$ for the scheme $(\mathcal P^{h,\ttau})$
with the additional right hand side 
$(\pi^h\vec f^{m+1}_{\mathcal P}, \vec\eta)$.}
\label{tab:torusPquad}
\end{table}%
As a comparison, we provide the corresponding computation for the scheme
$(\mathcal Q^{h,\ttau})$ in Table~\ref{tab:torusQquad}, where the same 
optimal convergence rates can be observed.
\begin{table}
\center
\begin{tabular}{|r|c|c|c|c|}
\hline
$J$ & $\displaystyle\max_{m=0,\ldots,M} |\vec x^m - \vec X^m|_0$ & EOC
& $\displaystyle\max_{m=0,\ldots,M} |\vec x^m - \vec X^m|_1$ & EOC 
\\ \hline
32  & 8.9663e-03 & ---  & 3.5842e-01 & ---  \\
64  & 2.2421e-03 & 2.00 & 1.7836e-01 & 1.01 \\
128 & 5.6058e-04 & 2.00 & 8.9071e-02 & 1.00 \\
256 & 1.4015e-04 & 2.00 & 4.4522e-02 & 1.00 \\
512 & 3.5037e-05 & 2.00 & 2.2259e-02 & 1.00 \\
\hline
\end{tabular}
\caption{Errors for the convergence test for (\ref{eq:solx})
over the time interval $[0,1]$ for the scheme $(\mathcal Q^{h,\ttau})$
with the additional right hand side $(\pi^h\vec f^{m+1}_{\mathcal Q}, 
\vec\eta)$.}
\label{tab:torusQquad}
\end{table}%

Let us next consider the initial surface given by the torus
\begin{equation*} 
\mathcal{S}(0) = 
\left\{ \vec z \in \bR^3 : 
\left( 1- |\vec z - (\vec z \cdot \vec\ek_2)\,\vec \ek_2| \right)^2 
+ (\vec z \cdot \vec\ek_2)^2=r^2 \right\} ,
\quad \text{where $0<r<1$}\,,
\end{equation*}
and denote by $T_r$ the time at which the solution 
$\mathcal{S}(t)$ of (\ref{eq:mcfS}) becomes singular. 
It is known, see \citet[Proposition~3]{SonerS93}, that there exists a critical
value $r_0 \in (0,1)$, such that for $0<r<r_0$
the solution shrinks to a circle at time $T_r$, 
while for $r_0<r<1$ it closes up the hole at time $T_r$. 
Furthermore, for $r=r_0$ these effects occur at the same time.

We first demonstrate the two different behaviours by repeating the experiments
in \citet[Figs.\ 2, 3]{aximcf}, see also \citet[Figs.\ 5, 6]{gflows3d}.
In particular, in the first experiment we let $r=0.7$.
As expected, we obtain a surface that closes
up towards a genus-0 surface, see Figure~\ref{fig:torusR1r07}.
As the discretization parameters we choose $J=512$ and $\ttau = 10^{-4}$.
\begin{figure}
\center
\includegraphics[angle=-90,width=0.35\textwidth]{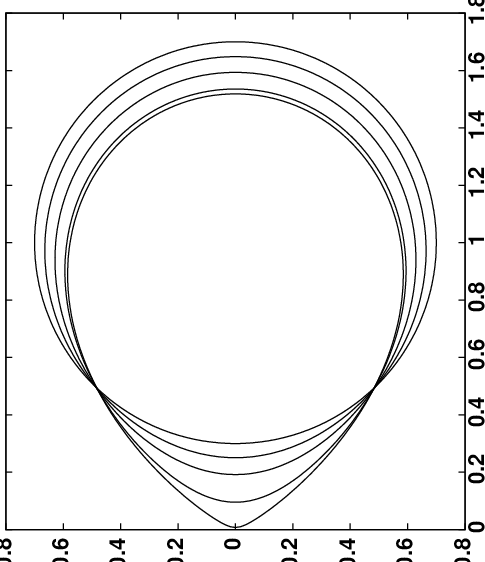}
\includegraphics[angle=-90,width=0.30\textwidth]{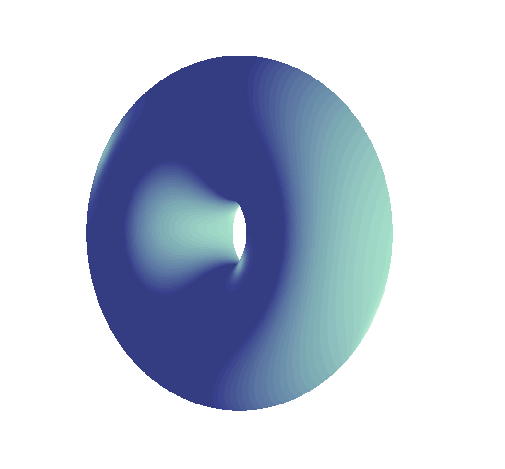}
\includegraphics[angle=-90,width=0.30\textwidth]{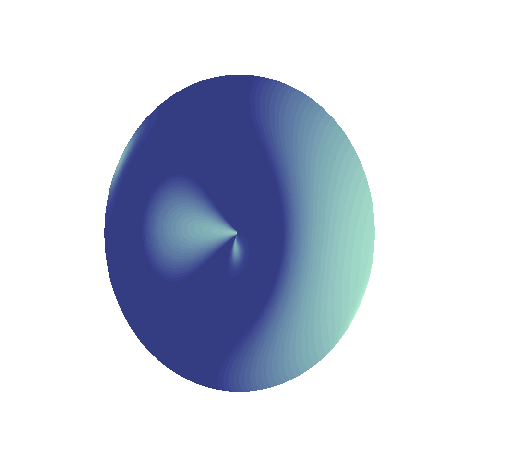}
\caption{Evolution for a torus with $r=0.7$. Plots are at times
$t=0,0.025,0.05,0.075,0.082$.
We also visualize the axisymmetric surface $\mathcal{S}^m$ generated by
$\Gamma^m$ at times $t=0$ and $t=0.082$, where $\Gamma^m = \vec X^m(I)$.}
\label{fig:torusR1r07}
\end{figure}%
For the second experiment we choose a torus with $r=0.5$ and obtain a shrinking 
evolution towards a circle. 
We show the evolution for the same discretization parameters 
in Figure~\ref{fig:torusR1r05}.
In addition, we show the evolution of the ratio
\begin{equation} \label{eq:ratio}
\ratio^m = \dfrac{\max_{j=1,\ldots,J} |\vec{X}^m(q_j) - \vec{X}^m(q_{j-1})|}
{\min_{j=1,\ldots,J} |\vec{X}^m(q_j) - \vec{X}^m(q_{j-1})|}
\end{equation}
over time. We can see that compared to the corresponding ratio plots in
\citet[Fig.\ 4]{aximcf}, the schemes $(\mathcal P^{h,\ttau})$ and
$(\mathcal Q^{h,\ttau})$ are performing
relatively well. In particular, even close to the singularity of the flow, 
when the surface shrinks to a circle and then vanishes, the
ratio appears to remain bounded.
\begin{figure}
\center
\includegraphics[angle=-90,width=0.35\textwidth]{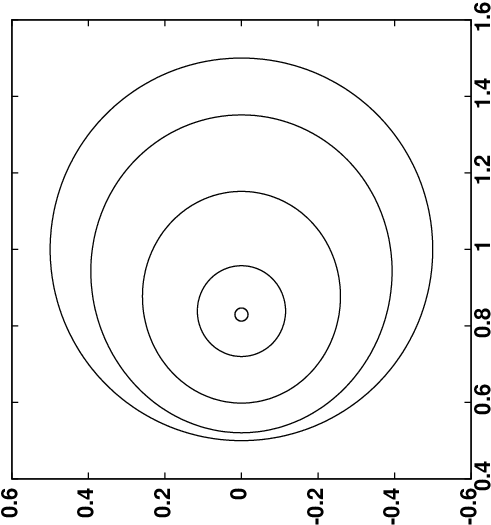}
\includegraphics[angle=-90,width=0.30\textwidth]{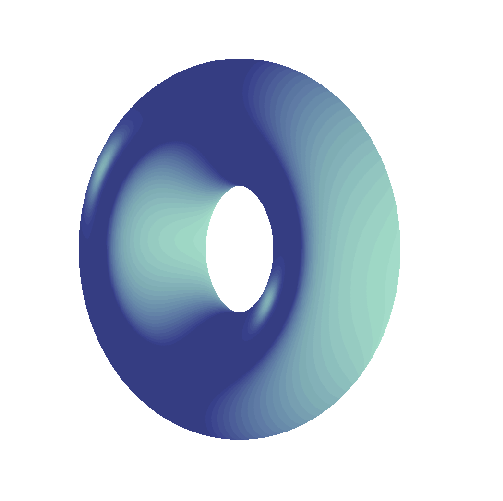}
\includegraphics[angle=-90,width=0.30\textwidth]{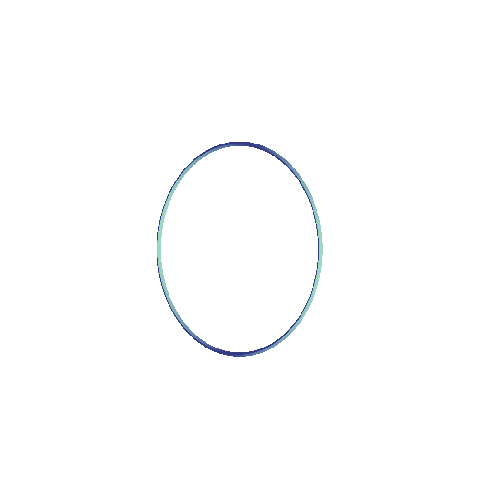}
\includegraphics[angle=-90,width=0.34\textwidth]{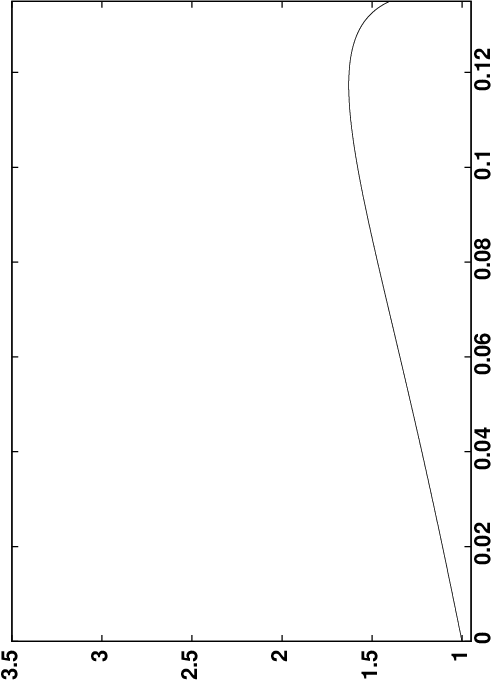}\quad
\includegraphics[angle=-90,width=0.34\textwidth]{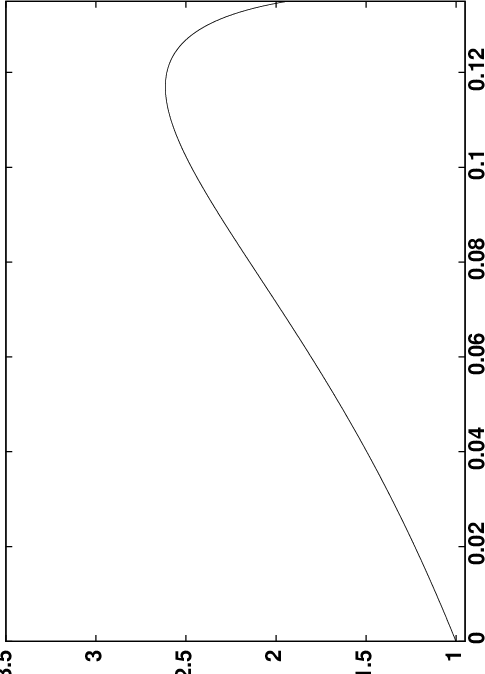}
\caption{Evolution for a torus with $r=0.5$. Plots are at times
$t=0,0.05,0.1,0.13,0.137$.
We visualize the axisymmetric surfaces generated by
$\Gamma^m$ at times $t=0$ and $t=0.137$.
Below we show plots of the ratio $\ratio^m$ over time for the schemes
$(\mathcal P^{h,\ttau})$, left, and $(\mathcal Q^{h,\ttau})$, right.
}
\label{fig:torusR1r05}
\end{figure}%

To the best of our knowledge, the precise value of the critical radius $r_0$ is
not yet known. However, we note that \citet[Thm.\ 2.1]{Ishimura93}, 
see also \cite{AharaI93}, gives a rigorous proof that 
$r_0 \geq \frac 2{3+\sqrt{5}} \approx 0.38$.
Numerical approaches to approximate $r_0$ have been presented in 
\cite{PaoliniV92} and \cite{Chopp94}, with the former giving an estimate
of $r_0 \approx 0.65$.
In what follows, we employ our scheme $(\mathcal P^{h,\ttau})$ in order to
obtain an accurate approximation of $r_0$ by repeating the 
above simulations for various values of $r$, with the help of a bisection 
method. In our experiments, and with the finer discretization parameters
of $J=2048$ and $\ttau=10^{-5}$, we observe that $r_0$
appears to lie in the interval $[0.64151,0.64152]$. We support this finding 
with the two simulations shown in Figure~\ref{fig:torus_r}.
We note that the stated interval
for the value of $r_0$ is confirmed when validating it with the finer
discretization parameters $J=4096$ and $\ttau=5\times10^{-6}$.
\begin{figure}
\center
\includegraphics[angle=-90,width=0.4\textwidth]{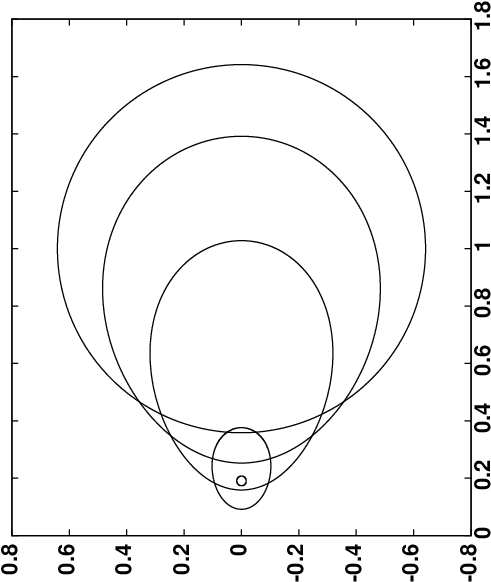}
\qquad
\includegraphics[angle=-90,width=0.4\textwidth]{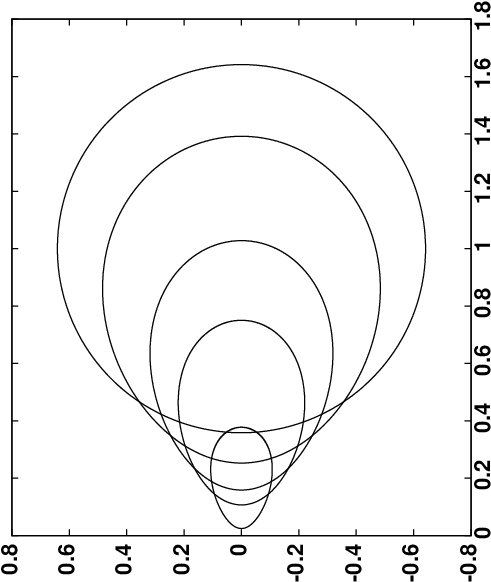}
\caption{
Evolution for a torus with $r=0.64151$ (left) and $r=0.64152$ (right). 
Plots are at times $t=0,0.1,0.2,0.25,0.29,0.298$ (left) and
$t=0,0.1,0.2,0.25,0.29$ (right).
}
\label{fig:torus_r}
\end{figure}%

Of particular interest in differential geometry are self-similarly shrinking
solutions to the mean curvature flow, \eqref{eq:mcfS}. 
It is clear by inspection, see also \cite{Angenent92}, that any self-similar
solution of \eqref{eq:mcfS} must be of the form
\begin{equation} \label{eq:selfsim}
\mathcal{S}(t) = \left[1-\frac t{\overline T_0}\right]^\frac12
\mathcal{S}(0)\,,
\end{equation}
where $\overline T_0>0$ is the time at which the solution to \eqref{eq:mcfS} 
becomes singular. 
The existence of a genus-1 solution of mean curvature flow of the form 
\eqref{eq:selfsim} was proved by \cite{Angenent92}.
The associated surface
$\mathcal{S}(0)$, in the case $\overline T_0 = 1$, is now frequently
called the Angenent torus, see e.g.\ \cite{Mantegazza11}.
The Angenent torus is axisymmetric, and one of its important properties 
is, that it is a critical point of 
Huisken's $F$-functional, see \cite{Huisken90},
defined for a hypersurface $\mathcal{S} \subset \bR^3$ as
\begin{equation} \label{eq:HuiskenF}
F_{\mathcal{S}}(\mathcal{S}) 
= \frac1{4\,\pi} \,\int_{\mathcal{S}} e^{-\frac14\,|\vec\id|^2} \dH{2}\,,
\end{equation}
where $\vec\id$ denotes the identity function, and
$\mathcal{H}^2$ is the two-dimensional Hausdorff measure in $\bR^3$.

In what follows, we would like to apply the techniques introduced in this paper
in order to compute approximations of the Angenent torus. To this end, we first
of all note that in the axisymmetric setting, the self-similar solution
\eqref{eq:selfsim} 
is generated by an evolving curve parameterized by 
$\vec y(t) : I \to \bRgeq \times \bR$ such that
\begin{equation} \label{eq:yselfsim}
\vec y(t) = \left[1-\frac{t}{\overline{T}_0}\right]^\frac12\,\vec y(0)\,.
\end{equation}
Clearly, \eqref{eq:yselfsim} is a solution of \eqref{eq:DDstrong3} 
if and only if $\vec y(0)$ satisfies the elliptic equation
\begin{equation} \label{eq:Klaus}
\frac1{2\,\overline{T}_0}\,\vec y \cdot \vec\ek_1\, |\vec y_\rho|^2\,\vec y + 
((\vec y \cdot \vec\ek_1)\, \vec y_\rho)_\rho 
- | \vec y_\rho |^2\,\vec\ek_1 = \vec 0 \,.
\end{equation} 
The natural finite element approximation of \eqref{eq:Klaus}, 
similarly to \eqref{eq:DDlinear}, is
given by: Find $\vec Y^h \in \Vh$ such that 
\begin{equation} \label{eq:DDY}
\mathcal{F}^h_{\overline{T}_0}(\vec Y^h) = \vec 0 \in \Vh\,, 
\end{equation}
where for $\alpha \in \bRplus$ and $\vec\chi \in \Vh$
we define $\mathcal{F}^h_\alpha(\vec\chi) \in \Vh$ via
\begin{equation} \label{eq:Fh}
\left(\mathcal{F}^h_\alpha(\vec \chi), \vec\eta \right)
= \frac1{2\,\alpha}
 \left((\vec \chi\cdot\vec\ek_1)\, \vec \chi , \vec\eta\,|\vec \chi_\rho|^2
\right) - \left( (\vec \chi\cdot\vec\ek_1)\,
\vec \chi_\rho,\vec\eta_\rho \right)
- \left( \vec\eta \cdot\vec\ek_1, |\vec \chi_\rho|^2\right)
\quad \forall\ \vec\eta \in \Vh\,.
\end{equation}
In practice, a solution to \eqref{eq:DDY} can be found with the help of a
damped Newton iteration, provided that a suitable initial guess is used. In
all our experiments for $\overline{T}_0=1$, the iteration converges in
fewer than 10 steps when 
the initial value for $\vec Y^h$ parameterizes a circle of radius $0.6$ 
centred around $2\,\vec\ek_1$. 
In Figure~\ref{fig:Angenent} we display some of the computed
approximations to the Angenent torus for different values of $J$.
\begin{figure}
\center
\mbox{
\includegraphics[angle=-90,width=0.3\textwidth]{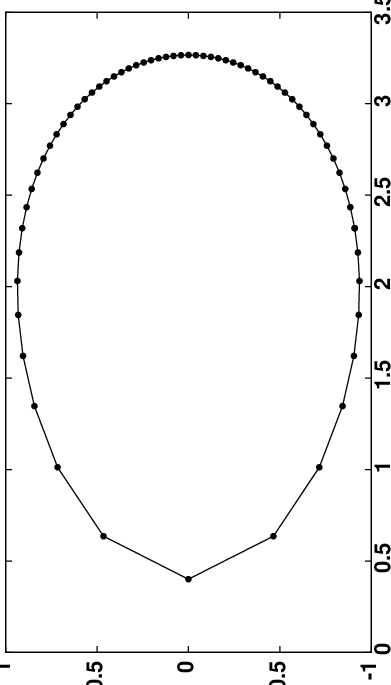}
\includegraphics[angle=-90,width=0.3\textwidth]{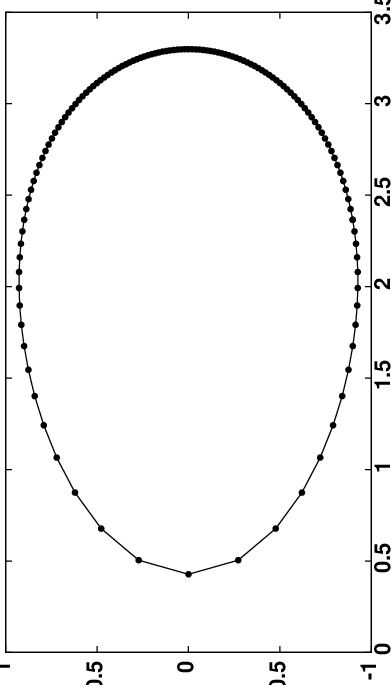}
\includegraphics[angle=-90,width=0.3\textwidth]{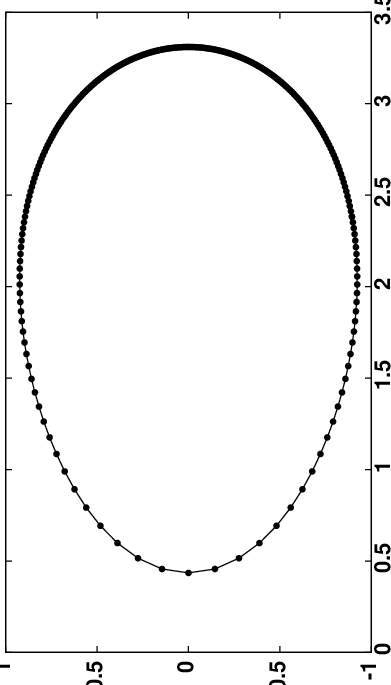}
}
\includegraphics[angle=-90,width=0.6\textwidth]{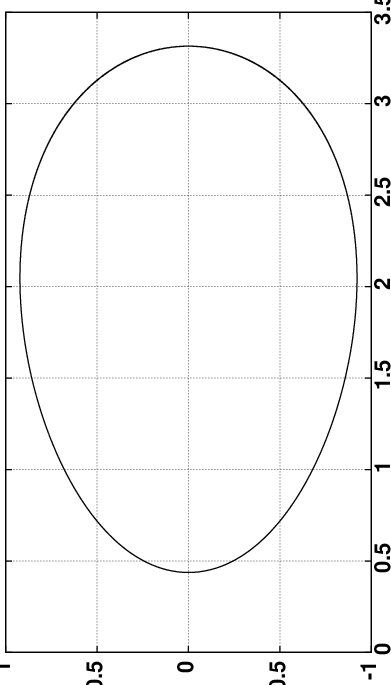}
\includegraphics[angle=-90,width=0.35\textwidth]{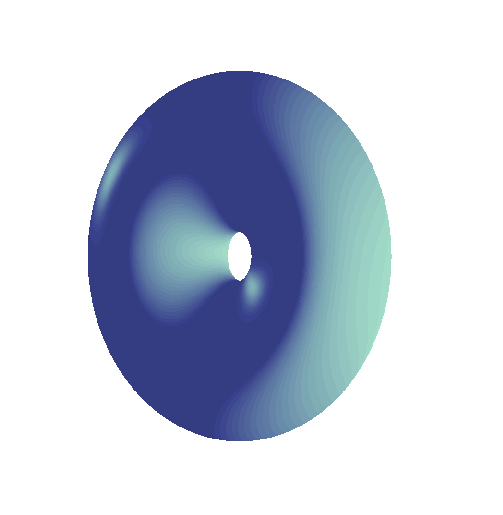}
\caption{
The solution $\vec Y^h$ to \eqref{eq:DDY} for $J=64,128,256$, above, and
for $J=8192$, below. We also visualize the axisymmetric surface generated by
$\vec Y^h(I)$, for $J=8192$.
}
\label{fig:Angenent}
\end{figure}%
In addition, in Figure~\ref{fig:selfsimsol} we present the evolution under 
mean curvature flow for the Angenent torus, 
by employing our scheme $(\mathcal P^{h,\ttau})$
with $J=4096$ and $\ttau=10^{-5}$. The presented plot 
nicely demonstrates the self-similar nature of the evolution, as well as the
fact that the extinction time for the Angenent torus is $\overline T_0=1$.
\begin{figure}
\center
\includegraphics[angle=-90,width=0.5\textwidth]{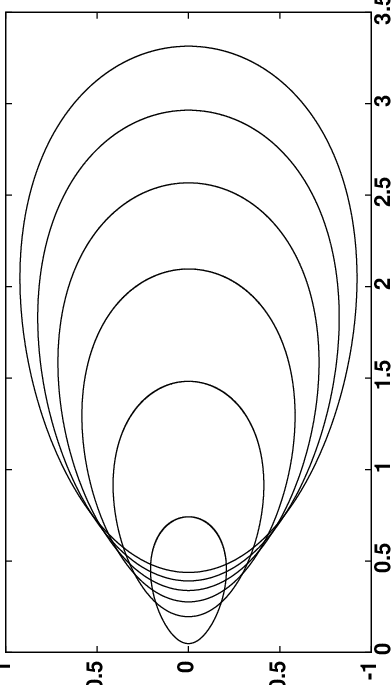}
\caption{
Evolution for the Angenent torus. Plots are at times $t=0,0.2,\ldots,0.8,0.95$.
}
\label{fig:selfsimsol}
\end{figure}%

For completeness we also compute highly accurate numerical approximations to
Huisken's $F$-functional 
\eqref{eq:HuiskenF}, as well as to the volume and the surface area of
the Angenent torus, with the help of our obtained solutions $\vec Y^h$.
To this end, let
\begin{align*} 
F(\vec Y^h) &
= \tfrac12 \left( \vec Y^h\cdot\vec\ek_1\,e^{-\frac14\,|\vec Y^h|^2}, 
|\vec Y^h_\rho| \right) ,\\ 
V(\vec Y^h) & = \pi \left( (\vec Y^h\,.\,\vec\ek_1)^2, 
[\vec Y^h_\rho]^\perp\,.\,\vec\ek_1\right) ,\quad
A(\vec Y^h) = 2\,\pi\left(\vec Y^h\,.\,\vec\ek_1 ,
|\vec Y^h_\rho|\right) ,
\end{align*}
where $(\cdot)^\perp$ denotes a clockwise rotation by $\frac{\pi}{2}$,
and where we have used the convention that 
$|\vec Y^h_\rho|^{-1}\,(\vec Y^h_\rho)^\perp$ 
denotes the outer normal to the curve
$\vec Y^h(I)$, see e.g.\ \cite{axisd} for details.
In Table~\ref{tab:Angenent} we display these numerical approximations, 
together with additional characteristic properties, 
for different values of $J$.
We remark that \eqref{eq:HuiskenF} for the Angenent torus has been
approximately computed in \cite{Berchenko-Kogan19}, with a value of
$1.85122$, which agrees well with our values for $F(\vec Y^h)$ reported in
Table~\ref{tab:Angenent}.
\begin{table}
\center
\begin{tabular}{|r|c|c|c|c|c|c|}
\hline
$\log_2 J$ & $F(\vec Y^h)$ & $V(\vec Y^h)$ & $A(\vec Y^h)$ & 
$\min_I \vec Y^h\cdot\vec\ek_1$ & $\max_I \vec Y^h\cdot\vec\ek_1$ & 
$\max_I \vec Y^h\cdot\vec\ek_2$ \\ \hline
16&1.8512166818&50.01714212&89.94051108&0.43712393&3.31470820&0.92171402\\
17&1.8512166742&50.01714302&89.94051299&0.43712396&3.31470825&0.92171401\\
18&1.8512166723&50.01714324&89.94051347&0.43712396&3.31470826&0.92171400\\
19&1.8512166718&50.01714329&89.94051359&0.43712397&3.31470827&0.92171400\\
20&1.8512166717&50.01714331&89.94051362&0.43712397&3.31470827&0.92171400\\
\hline
\end{tabular}
\caption{
Approximate values for Huisken's $F$-function, the enclosed volume, the surface
area and the dimensions of the Angenent torus, computed with the help of the
solutions $\vec Y^h$ to \eqref{eq:DDY}. 
}
\label{tab:Angenent}
\end{table}%

Finally, we are interested whether in the evolutions for the critical radius 
$r_0$ in
Figure~\ref{fig:torus_r}, the observed shapes become similar to scaled versions
of the Angenent torus. To this end, we develop the following criterion for
self-similarity of a given finite element approximation $\vec Z^h\in\Vh$. 
On recalling \eqref{eq:Fh}, 
we define 
the scale-invariant goodness of a self-similarity fit as
\begin{equation} \label{eq:alpha}
G(\vec Z^h) = \min_{\alpha\in\bRplus}
\frac{\left|\mathcal{F}^h_\alpha(\vec Z^h)\right|_0}
{\left( 1, |\vec Z^h_\rho|^2\right)}\,,
\end{equation}
where we observe that finding the minimizing value of $\alpha\in\bRplus$ in 
\eqref{eq:alpha} is trivial, since 
$\left|\mathcal{F}^h_\alpha(\vec Z^h)\right|_0^2$ is a quadratic function in
$\alpha^{-1}$.
In Figure~\ref{fig:torus_rG} we show the plots over time of the quantity $G$
for the two evolutions for the critical radius in Figure~\ref{fig:torus_r}. 
As expected, for both evolutions the value of $G$ is initially decreasing,
before it increases sharply as the extinction time is approached.
Throughout the evolutions it holds that $G(\vec X^m) > 0.17$.
In comparison, the four curves in Figure~\ref{fig:Angenent} all satisfy
$G(\vec Y^h) < 10^{-10}$.
Hence, overall, we conclude that the evolutions for the critical radius $r_0$ 
do become more similar to a scaled Angenent torus over time, without being
able to reach it exactly.
\begin{figure}
\center
\includegraphics[angle=-90,width=0.4\textwidth]{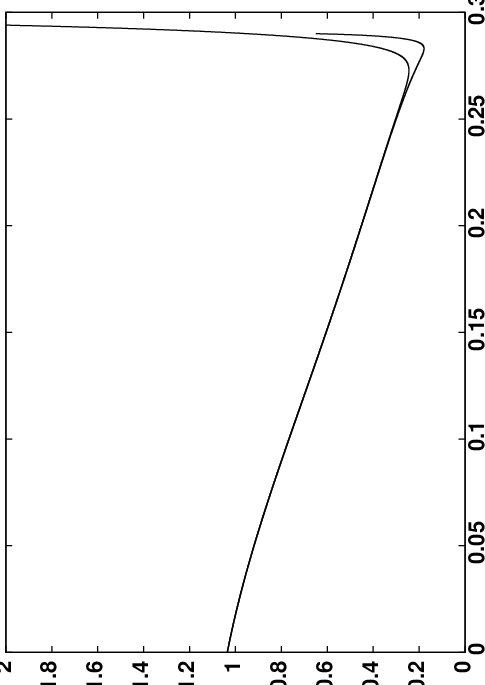} \qquad
\includegraphics[angle=-90,width=0.4\textwidth]{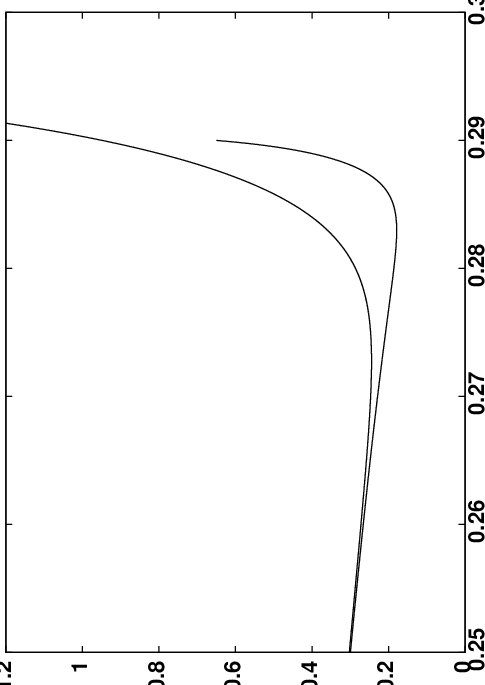}
\caption{A plot of $G(\vec X^m)$ over time for the two evolutions in
Figure~\ref{fig:torus_r}. On the right a plot over the time interval
$[0.25,0.3]$.
}
\label{fig:torus_rG}
\end{figure}%

Let us finish this section with an example for mean curvature flow of a genus-1 surface that is
generated from the initial data $\vec X^0$ parameterizing a
closed spiral. 
As can be seen from Figure~\ref{fig:spiral}, the spiral slowly untangles,
until the surface approaches a shrinking torus, that will once again shrink to
a circle.
For this experiment we use the discretization parameters 
$J=1024$ and $\ttau = 10^{-6}$, for $T=0.029$. 
\begin{figure}
\center
\newcommand\localwidth{0.24\textwidth}
\includegraphics[angle=-90,width=\localwidth]{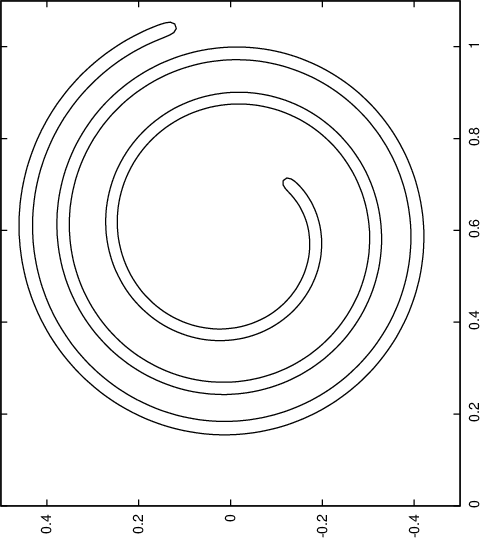}
\includegraphics[angle=-90,width=\localwidth]{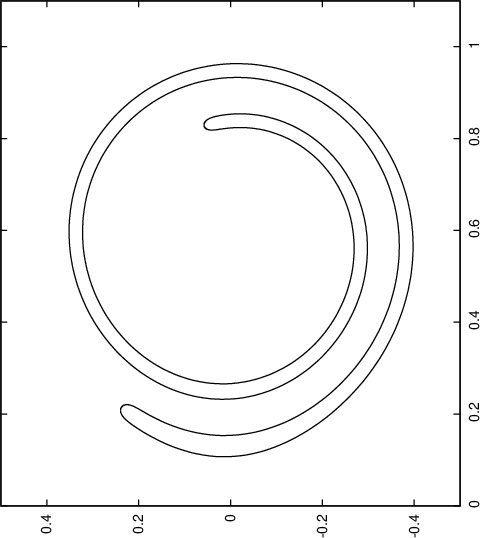}
\includegraphics[angle=-90,width=\localwidth]{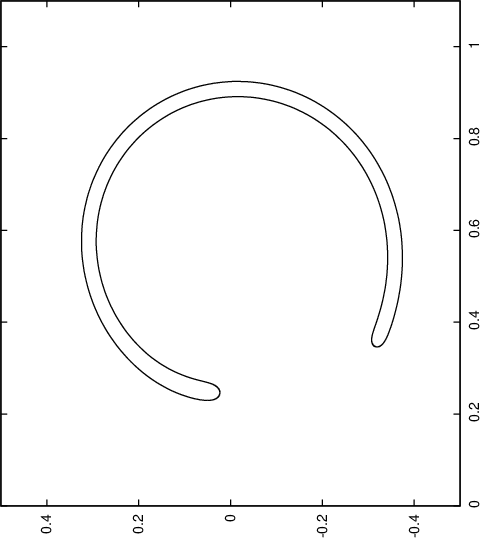}
\includegraphics[angle=-90,width=\localwidth]{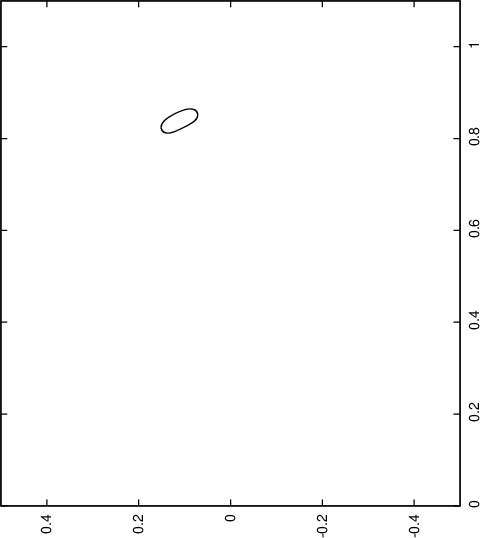}
\includegraphics[angle=-90,width=\localwidth]{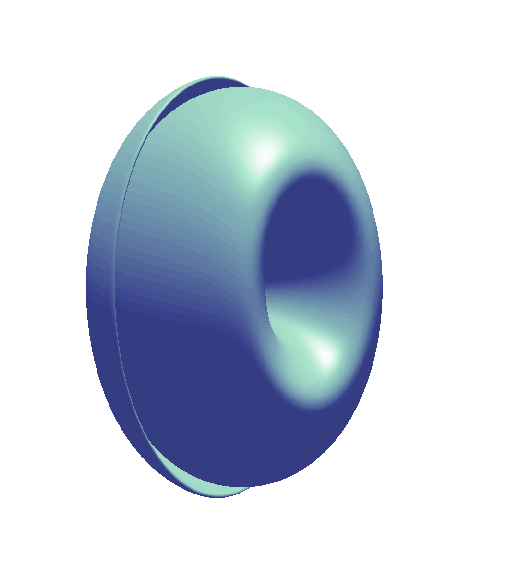}
\includegraphics[angle=-90,width=\localwidth]{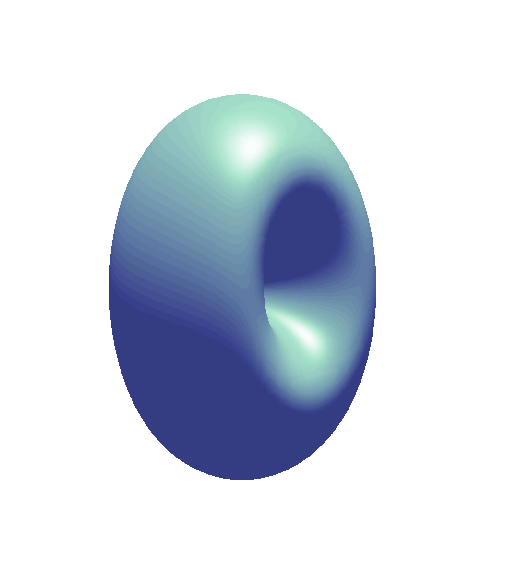}
\includegraphics[angle=-90,width=\localwidth]{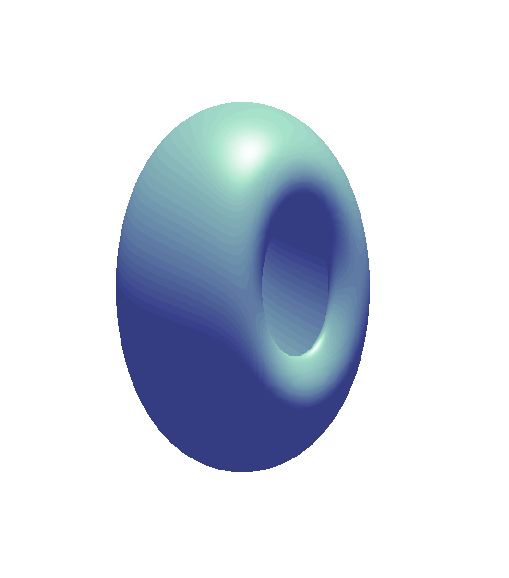}
\includegraphics[angle=-90,width=\localwidth]{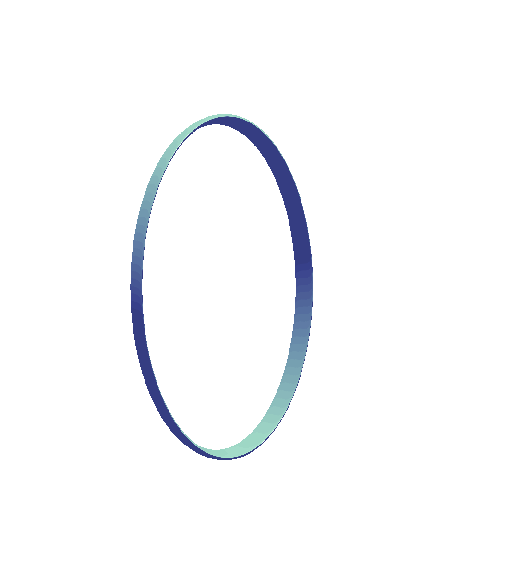}
\caption{Evolution for a genus-1 surface generated by a spiral.
Plots are at times $t=0,0.01,0.02,0.029$.
Below we visualize the axisymmetric surfaces generated by the curves.
}
\label{fig:spiral}
\end{figure}%

\subsection{Genus-0 surfaces}

We recall that the error bounds in Theorem~\ref{thm:ebP}
are only shown for the case of a closed curve. 
Nevertheless, in this section we want to consider
surfaces that are topologically equivalent to a sphere, and so $\vec x$
parameterizes an open curve with endpoints on the $x_2$--axis.

It is easy to show that a shrinking sphere with 
radius $[1 - 4\,t]^\frac12$, which corresponds to $\overline T_0 = \frac14$
and $\mathcal{S}(0)$ being the unit sphere in \eqref{eq:selfsim},
is a solution to (\ref{eq:mcfS}). 
In fact, the parameterization 
\begin{equation} \label{eq:solxpi}
\vec x(\rho, t) = 
[1 - 4\,t]^\frac12\, \begin{pmatrix}
\sin (\pi\,\rho) \\ \cos (\pi\,\rho)
\end{pmatrix}
\end{equation}
solves \eqref{eq:DDstrong3}. Hence we can compare e.g.\ $\vec x^{m+1}$ 
to the discrete solution $\vec X^{m+1}$ of \eqref{eq:DDlinear} and perform a
convergence experiment. As before, we choose $\ttau = h^2$,
for $h = J^{-1} = 2^{-k}$, $k=5,\ldots,9$. 
The results in Table~\ref{tab:spherePquad} indicate that despite the open 
curve case not being covered in Theorem~\ref{thm:ebP}, 
we still seem to observe the optimal convergence rates in practice.
\begin{table}
\center
\begin{tabular}{|r|c|c|c|c|c|}
\hline
$J$ & $\displaystyle\max_{m=0,\ldots,M} |\vec x^m - \vec X^m|_0$ & EOC &
$\displaystyle\max_{m=0,\ldots,M} |\vec x^m - \vec X^m|_1$ & EOC
 \\ \hline
32  & 8.0301e-04 & ---  & 8.9023e-02 & ---  \\
64  & 2.0079e-04 & 2.00 & 4.4572e-02 & 1.00 \\
128 & 5.0199e-05 & 2.00 & 2.2285e-02 & 1.00 \\
256 & 1.2550e-05 & 2.00 & 1.1139e-02 & 1.00 \\
512 & 3.1375e-06 & 2.00 & 5.5674e-03 & 1.00 \\
\hline
\end{tabular}
\caption{Errors for the convergence test for (\ref{eq:solxpi})
over the time interval $[0,0.125]$ for the scheme 
$(\mathcal P^{h,\ttau})$.}
\label{tab:spherePquad}
\end{table}%

We remark that the main reason we restricted our attention in
Section~\ref{sec:alt} to the case of closed curves, is that it is not clear
whether that alternative formulation is well-posed at the boundary. 
Let us give a formal justification for this observation: a smooth function 
satisfying (\ref{eq:axibc}), (\ref{eq:bc}) will have the property that 
$\vec x_t \cdot \vec\ek_1$ is small and 
$\frac{\vec x_{\rho}}{| \vec x_{\rho} |}$ behaves like $\pm \vec \ek_1$ close 
to $\partial I$. Using this information in (\ref{eq:DD2strong}), one sees that
\begin{equation} \label{eq:nonunif}
\frac{\vec x_{\rho \rho} \cdot \vec x_{\rho}}{| \vec x_{\rho}|^3} 
+ \frac{1}{\vec x \cdot \vec\ek_1} \approx 0 
\quad \mbox{ close to } \partial I\,,
\end{equation}
implying that $( \frac{1}{| \vec x_{\rho} |} )_\rho$ becomes large 
close to $\partial I$. This suggests that the approach using 
\eqref{eq:DD2strong} is not appropriate for describing the evolution of 
genus-0 surfaces. For the formulation \eqref{eq:newsystem}, on the other hand,
we would obtain
\begin{equation} \label{eq:unif}
\frac{\vec x_{\rho \rho} \cdot \vec x_{\rho}}{| \vec x_{\rho}|^3} 
= - \left( \frac{1}{| \vec x_{\rho} |} \right)_\rho
\approx 0 \quad \mbox{ close to } \partial I\,,
\end{equation}
which is clearly satisfied by \eqref{eq:solxpi}. 
In fact, the difference between the two formulations can be clearly seen in
practice. Let $(\mathcal Q^{h,\ttau}_{\partial_0})$ be the scheme
$(\mathcal Q^{h,\ttau})$ with $\Vh$ replaced by $\Vhpartialzero$, i.e.\ the
obvious generalisation of the scheme to the case of $I=(0,1)$.
We now compare the behaviour of this adapted scheme to $(\mathcal
P^{h,\ttau})$, with particular attention to the movement of the vertices near 
the boundary. To this end, let the initial surface be given by disc of 
dimension $9\times1\times9$. Under mean curvature flow, the disc shrinks to a 
nearly spherical shape, and then shrinks to a point, 
see Figure~\ref{fig:flatcigar}.
For the scheme $(\mathcal P^{h,\ttau})$, with discretization parameters 
$J=128$ and $\ttau = 10^{-4}$,
we observe that the distribution of vertices close to the boundary remains
uniform, as is to be expected from (\ref{eq:unif}). Further away
from the boundary the density of vertices increases, which leads to a 
moderate increase of the ratio \eqref{eq:ratio} over time. Close to the
extinction time that ratio reduces again.
The scheme $(\mathcal Q^{h,\ttau}_{\partial_0})$, on the other hand, exhibits
a very nonuniform distribution of vertices close to the boundary, with the
length of the elements increasing as the boundary is approached. Of course,
this is in line with the analysis in (\ref{eq:nonunif}).
Overall, this leads to a far more pronounced increase in the ratio 
\eqref{eq:ratio} over time, compared to the scheme $(\mathcal P^{h,\ttau})$.
\begin{figure}
\center
\includegraphics[angle=-90,width=0.4\textwidth]{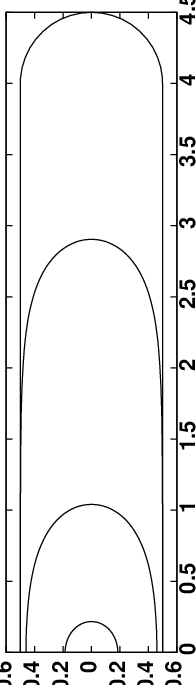} 
\includegraphics[angle=-90,width=0.24\textwidth]{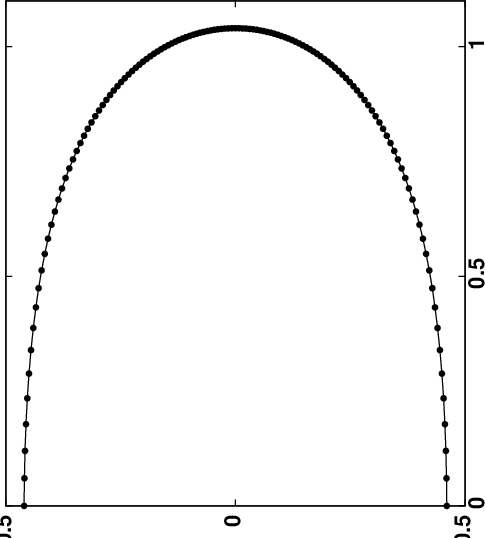} 
\includegraphics[angle=-90,width=0.24\textwidth]{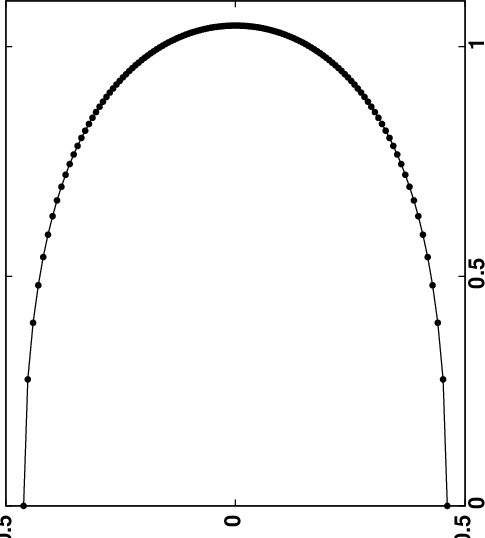} 
\includegraphics[angle=-90,width=0.4\textwidth]{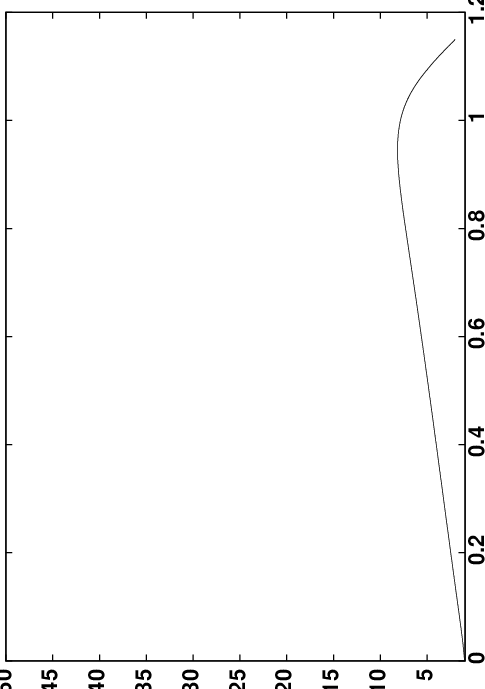} 
\includegraphics[angle=-90,width=0.4\textwidth]{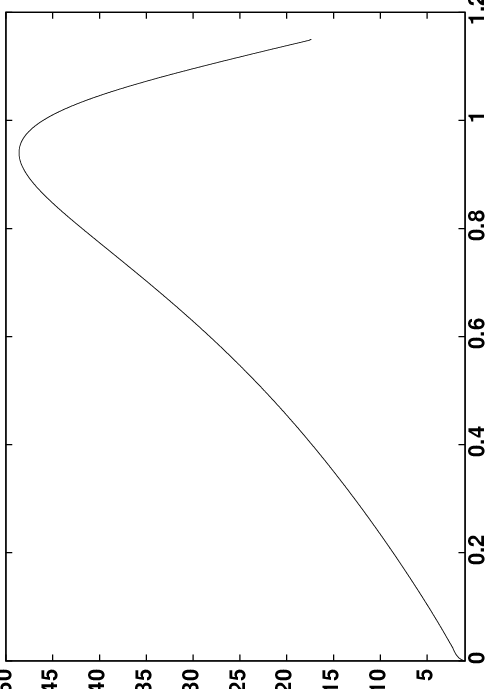}
\caption{
Evolution for a disc of dimension $9\times1\times9$.
The first row shows the solution at times $t=0,0.5,1,1.15$ for the scheme 
$(\mathcal P^{h,\ttau})$. Moreover, we present plots of the
distribution of vertices on $\Gamma^m$ at time $t=1$, for the scheme 
$(\mathcal P^{h,\ttau})$ (middle), as well as for the adapted scheme 
$(\mathcal Q^{h,\ttau}_{\partial_0})$ (right).
Below we show plots of the ratio $\ratio^m$ over time for 
$(\mathcal P^{h,\ttau})$, left, and for 
$(\mathcal Q^{h,\ttau}_{\partial_0})$, right.
}
\label{fig:flatcigar}
\end{figure}%

We end this section with an experiment for an initial dumbbell shape. 
For the experiment in Figure~\ref{fig:dumbbell} we used the
discretization parameters $J=512$ and $\ttau=10^{-4}$.
During the evolution the neck of the dumbbell is thinning, until this 
eventually leads to pinch-off, one of the singularities that can occur for mean
curvature flow.
\begin{figure}
\center
\includegraphics[angle=-90,width=0.14\textwidth]{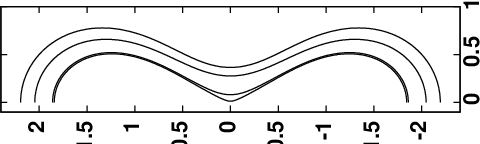} 
\qquad
\includegraphics[angle=-90,width=0.16\textwidth]{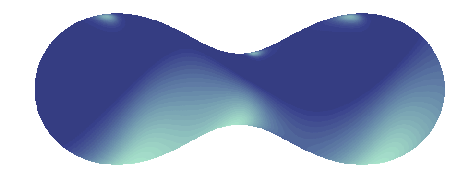} 
\includegraphics[angle=-90,width=0.16\textwidth]{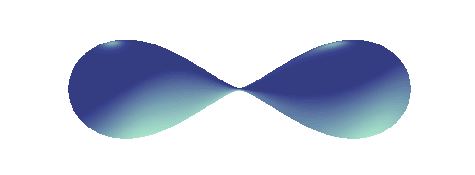} 
\caption{
Evolution for a dumbbell. Solution at times $t=0,0.05,0.1,0.104$.
We also visualize the axisymmetric surfaces generated by
$\Gamma^m$ at times $t=0$ and $t=0.104$.
}
\label{fig:dumbbell}
\end{figure}%

\section*{Acknowledgements}
We thank Yakov Berchenko-Kogan for helping us improve the
presentation and discussion of Figure~\ref{fig:torus_rG}.

\end{document}